\documentclass[12pt]{amsart}

\usepackage{amsthm,amssymb,amsmath,amstext,amsfonts}
\usepackage{enumitem,mathtools,pgfplots,pgfmath}
\pgfdeclarelayer{background}
\pgfsetlayers{background,main}
\usepackage[hyphens]{url}
\pgfdeclarelayer{background}
\pgfsetlayers{background,main}
\usepackage[margin=1in,letterpaper,portrait]{geometry}
\usepackage[latin1]{inputenc}
\usepackage{url,graphicx}
\usepackage{caption} 
\usepackage{mathtools,breqn}
\usepackage{stackengine}
\usepackage{scalerel}
\usepackage{dirtytalk}
\usepackage{bbm}
\usepackage{hhline}
\stackMath
\newcommand{\uP}{\underline{P}} 
\newcommand*{\Scale}[2][4]{\scalebox{#1}{$#2$}}%

\theoremstyle{plain}
\theoremstyle{definition}
\newtheorem{theorem}{Theorem}[section]
\newtheorem*{theorem-nonum}{Theorem}
\newtheorem{conjecture}[theorem]{Conjecture}

\newtheorem{lemma}[theorem]{Lemma}
\newtheorem{definition}[theorem]{Definition}
\newtheorem*{definition-nonum}{Definition}

\newtheorem{example}{Example}

\newtheorem{corollary}[theorem]{Corollary}

\DeclareMathAlphabet{\mathpzc}{OT1}{pzc}{m}{it}
\DeclareMathOperator{\red}{red}
\DeclareMathOperator{\cnt}{cnt}
\DeclareMathOperator{\drops}{drops}
\DeclareMathOperator{\peakSqSum}{peakSqSum}
\DeclareMathOperator{\des}{des}
\DeclareMathOperator{\adj}{adj}
\DeclareMathOperator{\Var}{Var}
\DeclareMathOperator{\subs}{subs}
\DeclareMathOperator{\Po}{Po}
\usepackage{multicol}
\usepackage{amsaddr}

\begin{document}

\title{Moments of permutation statistics and central limit theorems}

\author{Stoyan Dimitrov}
\address{University of Illinois at Chicago}
\author{Niraj Khare}
\address{Carnegie Mellon University in Qatar}
\email{sdimit@uic.edu, nkhare@cmu.edu}



\maketitle

\begin{abstract}
 We show that if a permutation statistic can be written as a linear combination of bivincular patterns, then its moments can be expressed as a linear combination of factorials with constant coefficients. This generalizes a result of Zeilberger. We use an approach of Chern, Diaconis, Kane and Rhoades, previously applied on set partitions and matchings. In addition, we give a new proof of the central limit theorem (CLT) for the number of occurrences of classical patterns, which uses a lemma of Burstein and H\"{a}st\"{o}. We give a simple interpretation of this lemma and an analogous lemma that would imply the CLT for the number of occurrences of any vincular pattern. Furthermore, we obtain explicit formulas for the moments of the descents and the minimal descents statistics. The latter is used to give a new direct proof of the fact that we do not necessarily have asymptotic normality of the number of pattern occurrences in the case of bivincular patterns. Closed forms for some of the higher moments of several popular statistics on permutations are also obtained.
 \end{abstract}

\section{Introduction}
If we have two combinatorial objects, $O$ and $p$, a natural question to ask is how many times does $p$ occur as part of $O$. Loosely speaking, we will refer to $p$ as the \emph{pattern}. Patterns in various combinatorial structures have been extensively studied in the past. This includes patterns in set partitions \cite{mansourBook}, trees \cite{trees}, Dyck paths \cite{dyckPaths} and permutations \cite{bonaBook,kitaev}. Many important statistics on these and other structures can be represented as linear combinations of patterns (i.e., the number of occurrences of certain patterns). 

Chern et al. \cite{partitions} showed that the moments (mean, variance and higher moments) of any such statistic on set partitions of $\{1,2,\ldots ,n\}$, can be written as a linear combination of shifted Bell numbers with coefficients that are polynomials in $n$. Their technique was also used for patterns in perfect matchings \cite{matchings}, in which case the moments of the corresponding statistics can be expressed as linear combinations of double factorials with constant coefficients. This is an analogous result since the total number of perfect matchings of given size is a double factorial, whereas the total number of set partitions of given size is a Bell number. 

In this paper, we adapt the approach of Chern et al. to permutations and obtain an analog to both of the mentioned results by showing that if a statistic on permutations of size $n$ can be written as a linear combination of bivincular patterns, then each of its higher moments can be expressed as a linear combination of shifted factorials of $n$ with constant coefficients. This generalizes the main theorem of Zeilberger in \cite{Z.I}, where he showed that each of these higher moments for the number of occurrences of any classical pattern is a polynomial in $n$ of a certain degree depending on the pattern. The same was proved for the variance of the number of occurrences of any vincular pattern \cite[Lemma 4.2]{hofer} and for an arbitrary moment of the number of classical pattern occurrences, when we sample from a conjugacy class of permutations \cite{gaetz}.

The obtained result (Theorem \ref{th:main}) allows us to derive exact formulas for the moments of various permutation statistics based on data for small values of $n$. Furthermore, we give new proofs to some central limit theorems for the number of permutation pattern occurrences.

\subsection{Central limit theorems for permutation patterns}
Assume that the two objects $O$ and $p$ are permutations. Some constraints on $p$ give us different types of patterns in permutations: consecutive, classical, vincular and bivincular. When $p$ is fixed and $O$ is selected at random from a set of permutations of given size, then we are naturally interested in the distribution of the number of occurrences of $p$, when the size of $O$ approaches infinity. 

Several previous works establish asymptotic normality of this distribution for different sets of patterns in permutations selected uniformly at random. For example, see Feller \cite[3rd ed., p.257]{feller} (for inversions), Mann \cite{mann} (for descents), Fulman \cite{fulman} (for both inversions and descents), Goldstein \cite{goldstein} and Borga \cite{jacopo} (for consecutive patterns), B\'{o}na \cite{bona} (for classical patterns) and Hofer \cite{hofer} (for vincular patterns). However, the number of occurrences of some simple bivincular patterns is not normally distributed (see Section \ref{subsec:bivinc}).

The recent works of Gaetz and Ryba \cite{gaetz} and Kammoun \cite{kammoun} establish normal limit laws on certain classes of permutations for classical and vincular patterns, respectively. In addition, Janson \cite{janson1,janson2} showed that the number of pattern occurrences is not normally distributed when we sample from the permutations avoiding a certain fixed pattern. Earlier, Janson, Nakamura and Zeilberger \cite{JNZ} initiated the study of the same general question. Two articles proving asymptotic normality for random permutations selected not according to the uniform measure are \cite{mallow,ferayEvans}. Finally, some important works \cite{baxter,zohar,JNZ} give central-limit theorems for certain joint-distributions of pattern occurrences. The listed articles use various approaches, from the method of moments \cite{Z_moments} to dependency graphs, Stein's method (see \cite[Section 3]{hofer} for overview of both methods) and the theory of U-statistics \cite[Chapter XI]{jansonBook}. 

We give a new proof of the central limit theorem (CLT) for the number of occurrences of any fixed classical pattern, first obtained by B\'{o}na \cite{bona}. In particular, we show that the lower bound for the variance of this number, which is a major part of his proof, follows from a lemma of Burstein and H\"{a}st\"{o} \cite{burstein}. We give a new simple interpretation of this lemma, which provides an intuitive explanation of why this CLT holds. We obtain a similar lemma with an analogous interpretation for the more general case of vincular patterns, which must hold since the CLT for an arbitrary vincular pattern was established by Hofer \cite{hofer}. Giving a combinatorial proof of the lemma, either in the case of classical or vincular patterns will be of great interest. Finally, we use a formula for the $r$-th moment of the minimal descent statistic that we obtain with the adapted approach of Chern et al., to give a new direct proof of the fact that we do not necessarily have asymptotic normality in the case of bivincular patterns. In particular, we show that the minimal descent statistic, which counts the number of occurrences of a simple bivincular pattern, has Poisson distribution. The most recent proof of this fact was given by Corteel et al. \cite{corteel}.

\subsection{Summary of the paper}
\label{subsec:summary}
The structure of the paper is as follows. In Sections \ref{sec:defs}, \ref{sec:aggr} and \ref{sec:hMoments} we adapt the definitions and tools developed in \cite{partitions, matchings} to permutations. Our main result, giving a closed form for the higher moments of a large class of permutation statistics, is Theorem \ref{th:main}. In Section \ref{sec:descents}, we demonstrate how one can use an important Corollary of our main result to obtain explicit formulas for any given higher moment of some simple permutation statistics, e.g., descents and minimal descents. Section \ref{sec:classical} contains the new proof of the CLT for the number of occurrences of classical patterns and Section \ref{sec:vincular} discusses how one can use the same approach to obtain a proof of the CLT in the more general case of vincular patterns. In Section \ref{subsec:bivinc}, we use one of the two formulas, obtained in Section \ref{sec:descents} to show that there are simple bivincular patterns whose number of occurrences does not have asymptotically normal distribution. Finally, in Section \ref{sec:lin_val_polynomials}, we give explicit formulas for the aggregates (resp., for the moments) of some permutation statistics in certain special cases, where a linearity of expectation arguments are directly applicable.

\section{Definitions and examples}
\label{sec:defs}
 Let $S_{n}$ be the set of all permutations of $[n]\coloneqq \{1,2, \ldots , n\}$, that is, the set of all bijections from $[n]$ to $[n]$. We will write any permutation $\pi$ using the one-line notation $\pi = \pi_{1}\pi_{2} \ldots \pi_{n}$, where $\pi(i)=\pi_{i}$ for all $i\in [n]$. Let $A(\pi)$ be the set of distinct pairs of integers $(u,v)$, such that $u$ occurs before $v$ in $\pi$. Formally, $A(\pi) \coloneqq \{(u,v)\mid u = \pi_{i}, v = \pi_{j},i<j\}$. 

To define statistics on permutations, we need the following definition of pattern, which is an analogue of those in \cite{partitions} for set partitions and \cite{matchings} for matchings.

\begin{definition} 
\label{def:pattern}
\leavevmode
\begin{enumerate}
    \item[(i)] A \emph{permutation pattern} $\underline{P}$ of length $k$ is a tuple $\underline{P} = (P,\boldsymbol{C}(\underline{P}),\boldsymbol{D}(\underline{P}))$, where $P = p_{1}\cdots p_{k}$ is a permutation of length $k$ and $\boldsymbol{C}(\underline{P})\subseteq [k-1]$, $\boldsymbol{D}(\underline{P})\subseteq [k-1]$ are two subsets.
    
    \item[(ii)] An \emph{occurrence} of the pattern $\underline{P}=(p_1p_2\cdots p_k,\boldsymbol{C}(\underline{P}), \boldsymbol{D}(\underline{P}))$ of length $k$ in $\sigma \in S_{n}$ is a tuple $s = (t_{1},t_{2}, \ldots , t_{k})$ with $t_{i} \in [n]$, such that:
    \begin{itemize}
        \item[a)] $t_1<t_2<\cdots<t_{k}$.
        \item[b)] $(t_{i},t_{j})\in A(\sigma)$, if and only if $(i,j)\in A(P)$.
        \item[c)] if $i\in \boldsymbol{C}(\underline{P})$, then $\sigma^{-1}(t_{p_{i+1}}) = \sigma^{-1}(t_{p_i})+1$, i.e., the positions of $t_{p_i}$ and $t_{p_{i+1}}$ in $\sigma$ are consecutive.
        \item[d)] if $i\in \boldsymbol{D}(\underline{P})$, then $t_{i+1} = t_{i}+ 1$, i.e., the values of $t_{i}$ and $t_{i+1}$ in $\sigma$ are consecutive.
    \end{itemize}
\end{enumerate}
    \end{definition}
Definition \ref{def:pattern} is equivalent to the definition of the so-called \emph{bivincular} patterns in permutations introduced by Bousquet-M\' elou et al. \cite {bivincular}. When $\boldsymbol{D}(\underline{P})= \emptyset$, then $\underline{P}$ is one of the \emph{vincular} patterns introduced by Babson and Steingr\' imsson \cite{babson}. When both $\boldsymbol{C}(\underline{P})= \emptyset$ and $\boldsymbol{D}(\underline{P})= \emptyset$, then $\underline{P}$ is a \emph{classical} pattern. For simplicity, when we have a classical pattern $\underline{P}$, we will refer to it just by writing the permutation $P$. For example, the classical pattern $\underline{P} = (132,\emptyset,\emptyset)$ will be denoted by $132$. When we have a vincular pattern, we will write $P$ with the positions $i$ and $i+1$ of $P$ being underlined for every $i\in\boldsymbol{C}(\underline{P})$. By that, we will indicate that these two numbers must be at consecutive positions in every occurrence of the pattern. For example, the vincular pattern $\underline{P} = (2314,\{2\},\emptyset)$ will be written as $2\underline{31}4$. When $\boldsymbol{D}(\underline{P})$ is non-empty, then we will use the two-line notation when referring to $\underline{P}$: if $P\in S_{k}$, the identity $id_{k}=12\cdots k$ will be on the top row with the numbers $j$ and $j+1$ of $id_{k}$ being overlined, for every $j\in\boldsymbol{D}(\underline{P})$. By that, we will indicate that the values of $t_{j}$ and $t_{j+1}$ must be consecutive in every occurrence $(t_{1}, \ldots , t_{k})$ of the pattern. On the bottom row, we will have the vincular pattern $\underline{P'} = (P,\boldsymbol{C}(\underline{P}),\emptyset)$ written in the usual one-line notation. For example, the bivincular pattern $\underline{P} = (43125,\{3\},\{1,4\})$ will be written as $\Vectorstack{\overline{12}3\overline{45} 43\underline{12}5}$.

We will write $s\in_{\underline{P}}\sigma$ if $s$ is an occurrence of $\underline{P}$ in $\sigma$. Throughout the paper, we will need the following definition.
\begin{definition}
Let $q=q_{1}\cdots q_{k}$ be a sequence of $k$ different numbers. The \emph{reduction} of $q$, denoted by $\red(q)$, is the unique permutation $\pi = \pi_{1}\cdots\pi_{k}\in S_{k}$, such that its elements are in the same relative order as the elements of $q$, i.e., $\pi_{i}<\pi_{j}$ if and only if $q_{i}<q_{j}$, for all $i,j\in [k]$. The permutation $\red(q)$ can be obtained by replacing the $i$-th smallest element of $q$ with $i$, for every $i\in [k]$.
\end{definition}

For example, $\red(523)=312$. Note that condition $(b)$ in the second part of Definition \ref{def:pattern} implies that if $s\in_{\underline{P}}\sigma$ and the elements of $s$ form the subsequence $q$ of $\sigma$, then $\red(q)=P$, i.e., the relative order of the numbers of the permutation $P$ and the numbers of $s$ in $\sigma$ is the same.
\leavevmode \newline

\textbf{Examples (occurrence of patterns):} \leavevmode

\begin{itemize}
    \item[1.] $(t_1,t_2,t_3)=(3,4,5)$ is an occurrence of $132 = (132,\emptyset,\emptyset)$ in $\sigma = 31524$, since $\red(354)=132$. 
    \item[2.] $(t_{1},t_{2},t_{3},t_{4})=(2,3,5,7)$ is an occurrence of $\underline{32}14 = (3214,\{1\},\emptyset)$ in  $\sigma = 4536217$ since $\red(5327)=3214$ and the positions of $t_{p_{1}}=t_{3}=5$ and $t_{p_{2}}=t_{2}=3$ in $\sigma$ are consecutive. 
    \item[3.] $(t_{1},t_{2},t_{3},t_{4})=(1,3,5,6)$ is an occurrence of $\scaleto{\Vectorstack{12\overline{34} 4\underline{31}2}}{25pt} = (4312,\{2\},\{3\})$ in  $\sigma = 625143$ since $\red(6513)=4312$, the positions of $t_{p_{2}}=t_{3}=5$ and $t_{p_{3}}=t_{1}=1$ in $\sigma$ are consecutive and the values of $t_{3}=5$ and $t_{4}=6$ in $\sigma$ are consecutive.  
\end{itemize}

The number of occurrences of the pattern $\underline{P}$ in $\sigma$ will be denoted by $\cnt_{\underline{P}}(\sigma)$. In the literature, usually a \emph{permutation statistic} is a function $T:S\to \mathbb{N}$, where $S=\bigcup_{i=1}^{\infty} S_{n}$. In this paper, when we write $\emph{statistic}$ or $\emph{simple statistic}$, we will refer to two classes of such functions defined below.

\begin{definition} \leavevmode

\begin{itemize}
        \item[(i)] A simple statistic is defined by a pattern $\underline{P}$ of length $k$ and a valuation function $Q(s,w)=Q_1(s)Q_2(w)$, which is a product of two polynomials $Q_1, Q_2\in\mathbb{Z}[y_{1},\ldots ,y_{k},m]$. If $\sigma\in S_n$ and $s = (t_{1},t_{2}, \ldots , t_{k}) \in_{\underline{P}}\sigma$, such that $\sigma(w_i)=t_i$, for all $i \in [k]$, then write $Q(s,\sigma^{-1}(s)) = Q_1(s)Q_2(\sigma^{-1}(s)) = Q_1\mid_{y_{i} = t_{i},m=n} Q_2\mid_{y_{i} = w_{i},m=n}$. Let 
    \begin{equation*}
        f(\sigma) = f_{\underline{P},Q}(\sigma)\coloneqq \sum_{{s\in_{\underline{P}}\sigma}
       }Q(s,\sigma^{-1}(s)) = \sum_{{s\in_{\underline{P}}\sigma}
       }Q_1(s)Q_2(\sigma^{-1}(s)).
    \end{equation*}
    
\noindent Let the degree of a simple statistic $f_{\underline{P},Q}$, denoted $d(f)$, be the sum of twice the length of $P$ and the degree of $Q$, which is the sum of the degrees of $Q_1$ and $Q_2$. 

\item[(ii)] A \emph{statistic} is a finite $Q$-linear combination of simple statistics. The degree of a statistic is defined to be the minimum, over all such representations, of the maximum degree of any of the included simple statistics.
\end{itemize}
\end{definition}

\textbf{Examples.}
\begin{itemize}
    \item[] $\cnt_{\underline{P}}\coloneqq f_{\underline{P},1}(\sigma)=\sum_{s\in_{\underline{P}}\sigma}1$, which counts the number of occurrences of the pattern $\underline{P}$ in $\sigma$, is a simple statistic for any pattern $\underline{P} = (P,\boldsymbol{C}(\underline{P}),\boldsymbol{D}(\underline{P}))$, with valuation function $Q=1$. If $P$ is of length $k$, then the degree of the statistic is $d(\cnt_{\underline{P}}) = 2k$. The first three examples we give below are of this kind, for $\underline{P}$ being classical, vincular and bivincular pattern (for which $\boldsymbol{C}(\underline{P}) \neq \emptyset$ and $\boldsymbol{D}(\underline{P}) \neq \emptyset$), respectively. \\
    \item[1.] \emph{Number of occurrences of $1324$.} 
    $$\cnt_{1324}(\sigma)=f_{1324,1}(\sigma) = \sum_{{s\in_{1324}\sigma}
       }1$$
    is the number of occurrences of the classical pattern $(1324,\emptyset,\emptyset)=1324$ in $\sigma$. This is the only classical pattern of length less than five for which the sequence of the number of permutations avoiding it, for different values of $n$, has not been enumerated yet. Two recent works related to this problem are \cite{bevan1324,mansour1324}.
    
    \item[2.] \emph{Number of double ascents}.
    \begin{equation*}
        \cnt_{\underline{123}}(\sigma) = \sum_{{s\in_{\underline{123}}\sigma}
       }1
    \end{equation*}
    is the number of occurrences of the vincular pattern $(123,\{1,2\},\emptyset)=\underline{123}$ in $\sigma$. The vincular patterns for which $\boldsymbol{C}(\underline{P})=[k-1]$ are called \emph{consecutive}. The generating function and the distribution of this statistic, as well as of $\cnt_{\underline{P}}$ for other vincular patterns of this kind were investigated in \cite{elizaldeNoy}.
    
    \item[3.] \emph{Number of occurrences of $\scaleto{\Vectorstack{1\overline{23} 3\underline{12}}}{25pt}$}.
    
    \begin{equation*}
        \cnt_{\scaleto{\Vectorstack{1\overline{23} 3\underline{12}}}{16pt}}(\sigma) = \sum_{{s\in_{\scaleto{\Vectorstack{1\overline{23} 3\underline{12}}}{11pt}}\sigma}
       }1
    \end{equation*}
    is the number of occurrences of the bivincular pattern $(312,\{2\},\{2\})$ in $\sigma$. It was shown in \cite{eriksen} that the number of permutations in $S_n$ with $k$ occurrences of this pattern is equal to the number of matchings on $[2n]$ with $k$ right nestings and no left nestings. 
    
    \item[4.] \emph{Descent drop}.
    \begin{equation*}
    \drops(\sigma) = \sum\limits_{\sigma_{i}>\sigma_{i+1}} \sigma_{i}-\sigma_{i+1} = \sum_{(t_{1},t_{2})\in_{\underline{21}}\sigma} t_{2}-t_{1}
    \end{equation*}
    is a simple statistic corresponding to the pattern $(21,\{1\},\emptyset)$ with valuation function $Q(s,w)=Q_{1}(s)Q_{2}(w)$, where $Q_{1}(s)=Q_{1}(t_{1},t_{2})=t_{2}-t_{1}$ and $Q_{2}(w)=1$. Thus, $deg(Q)=1$ and $d(\drops)=5$. Petersen and Tenner \cite{bridget} showed that this statistic is equidistributed with the statistic $dp(\sigma) = \sum_{\sigma(i)>i}\sigma(i)-i$, which they call \say{depth}. The depth of a permutation is half of another important statistic called \say{total displacement} or \say{Spearman's disarray}, whose generating function was found in \cite{petersonGF}. 
    
    \item[5.] \emph{Sum of peak squares}.
    \begin{equation*} 
        \peakSqSum(\sigma) = \sum_{\sigma(i-1)<\sigma(i)>\sigma(i+1)}\sigma(i)^{2}= \sum_{(t_{1},t_{2},t_{3})\in_{\underline{132}}\sigma} t_{3}^{2} + \sum_{(t_{1},t_{2},t_{3})\in_{\underline{231}}\sigma} t_{3}^{2}
    \end{equation*}
    is a statistic, which is a sum of the two simple statistics $f_{1} = f_{\underline{132},t_{3}^{2}}$ and $f_{2}= f_{\underline{231},t_{3}^{2}}$. Thus, $d(\peakSqSum)=max(d(f_{1}),d(f_{2})) = 8$. Two articles investigating the number of interior peaks and the number of permutations with a given set of peak values, called \say{pinnacle set}, are \cite{peaks} and \cite{bridget2}, respectively. To the best of our knowledge, the sum of the peaks and the sum of the squares of the peaks have not been yet investigated, despite of the recent interest in pinnacle sets \cite{lopez,sagan, rusu}. 
\end{itemize}
In the next two sections, we will show that the moments of all statistics are also statistics, as defined above, and we will give closed forms for one of these moments for each of the statistics above.

\section{Aggregates  of permutation statistics}
\label{sec:aggr}
We are often interested in the expected value $\mathbb{E}(f)$ of the permutation statistic $f$, for a permutation chosen uniformly at random from $S_{n}$. Obviously, we have $\mathbb{E}(f) = M(f,n)/n!$, where \[
M(f,n)\coloneqq\sum_{\sigma\in S_{n}}f(\sigma).
\] 
In this section, we show that the aggregate $M(f,n)$ is a linear combination of factorials with constant coefficients. This is an analogue of the results in \cite{partitions} for aggregates of set partition statistics and those in \cite{matchings} for aggregates of statistics on matchings. To deal with the constraints caused by $\boldsymbol{C}(\uP)$ and $\boldsymbol{D}(\uP)$, we use the same technique to compress numbers used in both of these articles. 
 \begin{theorem} \label{th:simpleStat}  
  Let $f_{\uP, Q}$ be a simple statistic of degree $m$ associated with the pattern $\uP$ of length $k$ and the valuation polynomial $Q(s,w)=Q_1(s)Q_2(w)$. Assume that $c=|\boldsymbol{C}(\uP)|$ and $d=|\boldsymbol{D}(\uP)|$. Then 
  \begin{eqnarray} \label{polynomial-form-exten} 
   M(f_{\uP, Q}, n) = R(n) (n-k)! 
  \end{eqnarray} 
where $R(x)$ is a polynomial of degree no more than $m-c-d$. 
Equivalently for $n \geq k$, $M(f,n)$ can be expressed as a linear combination of shifted factorials with constant coefficients, i.e., 
\begin{eqnarray} \label{T-ext-form} 
 M(f_{\uP, Q}, n) =  \left\{ 
\begin{array}{ll} 
0 & n < k \\ 
\displaystyle \sum_{i= 0 }^{m-c-d} c_i (n-k+i)!  & n \geq k 
\end{array} \right. ,
\end{eqnarray} 
for some constants $c_i \in \mathbb{Q}$. 
\end{theorem}

\begin{proof} 
Let $V_{n,k}:=\{(t_1,t_2,\cdots,t_k)\in [n]^k \mid 1\leq t_1<t_2<\cdots<t_k\leq n\}$ be the set of increasing vectors of $k$ numbers in $[n]$. For simplicity, fix $n$ and $k$ and let $T\coloneqq V_{n,k}$. Note that if $s\in_{\underline{P}} \sigma$ for some $\sigma \in S_n$, then $s\in T$. Let us also define $W:=\{(w_1,w_2,\ldots,w_k)\in [n]^k \mid \text{for all }i,j \in [k],\text{ if }w_i = w_j, \text{then }i=j\}$. Note that $|T| = \binom{n}{k}$ and $|W| = n(n-1)\cdots (n-k+1)$. We have
 \begin{eqnarray*}
  M(f_{\uP, Q}, n)  = 
  \sum_{\sigma \in {S}_{n}} f_{\uP, Q}(\sigma) &  = & \sum_{\sigma \in {S}_{n}} \sum_{s \in_{\uP} \sigma} Q_1(s)Q_2(\sigma^{-1}(s)) \\
  & = & \sum_{s \in T}\sum_{\substack{{\sigma \in S_n}\\{s \in_{\uP}\sigma}}}Q_1(s)Q_2(\sigma^{-1}(s))= \sum_{s \in T}Q_1(s)\sum_{\substack{{\sigma \in S_n}\\{s \in_{\uP}\sigma}}}Q_2(\sigma^{-1}(s)).
 \end{eqnarray*} 
For any $s \in T$, let $G(s):=\{\sigma \in S_n \mid s\in_{\uP} \sigma\}$ and for any $w\in W$, let $H(w):=\{\sigma \in S_n \mid (t_{1},\ldots ,t_{k})\in_{\uP} \sigma,\text{ where } \sigma^{-1}(t_i)=w_i \text{, for all } i\in [k]\}$. In addition, for any $s\in T$ and $w\in W$, let $Z(s,w):=\{\sigma \in S_n \mid s=(t_{1},\ldots ,t_{k})\in_{\uP} \sigma,\text{ } \sigma^{-1}(t_i)=w_i \text{ for all } i\in [k]\}$.  Clearly, for any $s\in T$ and $w, w' \in W$ for which $w\neq w'$, we have $Z(s,w)\cap Z(s,w')=\emptyset$. Also, note that $G(s)=\cup_{w\in W}Z(s,w) $. Hence, we can rewrite the above equations in the following way:
 \begin{eqnarray*}
  M(f_{\uP, Q}, n) & = & \sum_{s \in T}Q_1(s)\sum_{\substack{{\sigma \in S_n}\\{s \in_{\uP}\sigma}}}Q_2(\sigma^{-1}(s)) = \sum_{s \in T}Q_1(s)\sum_{\sigma \in G(s)}Q_2(\sigma^{-1}(s))\\
  & = & \sum_{s \in T}Q_1(s)\sum_{\sigma \in \cup_{w\in W}Z(s,w)}Q_2(w) = \sum_{s \in T}Q_1(s)\left(\sum_{w\in W}Q_2(w)\sum_{\sigma \in Z(s,w)}1\right)\\
 \end{eqnarray*} 
Consider any fixed vector of values $s\in T$ and a vector of positions $w\in W$. If $Z(s,w)\neq \emptyset$, then $|Z(s,w)|=(n-k)!$ since the remaining $n-k$ values, except those in $s$, can be arranged in all the possible ways at the remaining $n-k$ positions, which are not in $w$. Furthermore, if we define $T' \coloneqq \{s\in T \mid G(s)\neq \emptyset\}$ and $W' \coloneqq \{w\in W \mid H(w)\neq \emptyset\}$, then observe that the values in any $s\in T'$ can be at the positions determined by any $w\in W'$ and vice versa. In other words, $Z(s,w)\neq \emptyset$, if and only if $s\in T'$ and $w\in W'$. Therefore, 
\begin{eqnarray*}
M(f_{\uP, Q}, n) &=& \sum_{s \in T}Q_1(s)\left(\sum_{w\in W}Q_2(w)\sum_{\sigma \in Z(s,w)}1\right)\\
&=& (n-k)! \left(\sum_{s \in T'}Q_1(s)\right)\left(\sum_{w\in W'}Q_2(w)\right). 
\end{eqnarray*}
Consider $s \in T$ and $w \in W$, such that $Z(s,w)\neq \emptyset$. Now, we will use the compression technique, which relies on the following observation: Since $|D(\underline{P})|=d$, every subset of $[n-d]$ of $k-d$ different numbers corresponds to a set of values $s\in T'$ and the correspondence is one-to-one. Formally, let us call $i+1$ a \emph{follower}, if $i  \in\boldsymbol{D}(\underline{P})$ and a \emph{non-follower}, if $i  \notin\boldsymbol{D}(\underline{P})$. If $g(i)\in [k]$ denotes the index of the $i$-th non-follower, then let $y_i\coloneqq t_{g(i)}-(g(i)-i)$. Then, the vector $s\in T'$ determines uniquely the vector $(y_1, \ldots, y_{k-d})$ and one can see that $y_{u}<y_{v}$, if $u<v$. Indeed, it suffices to show this for $v=u+1$. In this case we have $y_{u+1}=t_{g(u+1)}-(g(u+1)-(u+1))>t_{g(u+1)-1}-(g(u+1)-(u+1))$, but we must have that $t_{g(u+1)-1}=t_{g(u)}+(g(u+1)-g(u)-1)$, because all the numbers between $g(u)$ and $g(u+1)$ are followers. Thus, $y_{u+1}>t_{g(u)}+(g(u+1)-g(u)-1)-(g(u+1)-(u+1)) = t_{g(u)}-(g(u)-u) = y_{u}$. Conversely, for any $(y_1, \ldots, y_{k-d})\in V_{n-d,k-d}$, the vector  $(t_1, \ldots, t_k)$ is uniquely determined, since $t_{j}=y_i+j-i$, where $j$ is the index of the $i$-th non-follower and $t_{j}=t_{j-1}+1$, if $j$ is an index of a follower. Thus $Q_1$ can be viewed as a polynomial in $y_1, \ldots, y_{k-d}$ and $n$. 

We can proceed in the same way for $W'$ and $\boldsymbol{C}(\underline{P})$. The only difference is that the elements of any $w\in W'$ are not necessarily in increasing order. However, the elements of $\overline{w} = (w_{P^{-1}}(1),\ldots ,w_{P^{-1}}(k))$ are always in increasing order and the map $w\mapsto\overline{w}$ is a bijection. Thus, using this map, we can get a set $W''\subseteq W$, such that there is a bijection between $W'$ and $W''$ and a bijection between $W''$ and $V_{n-c,k-c}$ (by the compression technique). Hence there is a bijection between $W'$ and $V_{n-c,k-c}$ and $Q_2$ can be viewed as a polynomial in $x_1, \ldots, x_{k-c}$ and $n$, where $(x_1, \ldots, x_{k-c})\in V_{n-c,k-c}$. Therefore, we have
\begin{eqnarray*} \label{tildeQ} 
M(f_{\uP, Q}, n)=(n-k)!\sum_{\scaleto{(y_1, \ldots, y_{k-d})\in V_{n{-}d,k{-}d}}{7.5pt}}  
\tilde{Q_1}(y_1, \ldots, y_{k-d}, n)  \sum_{\scaleto{(x_1, \ldots, x_{k-c})\in V_{n{-}c,k{-}c}}{7.5pt}}   
\tilde{Q_2}(x_1, \ldots, x_{k-c}, n)      
\end{eqnarray*} 
 for some polynomials $\tilde{Q_1}$ and $\tilde{Q_2}$ of the same degree as $Q_1$ and $Q_2$, respectively. The product of the two sums above yields a polynomial in $n$ of degree at most the sum of the following two terms: the maximum possible degree of $n$ in the product $\binom{n-d}{k-d} \tilde{Q_1}$ and the maximum possible degree of $n$ in the product $\binom{n-c}{k-c} \tilde{Q_2}$. Therefore, the degree of the product is at most $k-d+deg(Q_1)+k-c+deg(Q_2)=(deg(Q_1)+deg(Q_2)+2k)-c-d = m-c-d$, since $m=d(f_{\underline{P},Q}) = (deg(Q_1)+deg(Q_2)+2k)$. 
 
To see Equation \eqref{T-ext-form}, let $g_i(n)$ be a polynomial in $n$ defined by 
$g_i(n) = {(n-k+i)!} / (n-k)!$. Then $g_i$ is of degree $i$, and hence $\{ g_i(n)\}_{i=0}^\infty $ form a 
basis of $\mathbb{Q}[n]$.
It follows that any polynomial of degree $i$ can be written as a linear combination of $g_0(n), \dots, g_{i}(n)$. This 
implies Equation \eqref{T-ext-form}. 
\end{proof}
Next, we consider any general statistic. Recall that a statistic is a $\mathbb{Q}$-linear combination of simple statistics. 
\begin{theorem} \label{general-sum} 
 For any statistic $f$ of degree $m$, there is a positive integer $L \leq \frac{m}{2}$, such that for all $n \geq L$, 
 \begin{eqnarray} \label{eq:statPolyForm}
   M(f, n)=  U(n) (n-L)!,
 \end{eqnarray} 
where $U(n)$ is a polynomial of degree no more than $m+L$. Equivalently, if $n \geq L $,
\begin{eqnarray}  \label{T-form2} 
 M(f, n)= \sum_{-L\leq i \leq m} \alpha{_i} (n+i)!, 
\end{eqnarray}
for some constants $\alpha{_i}  \in \mathbb{Q}$. 
\end{theorem}
\begin{proof} 
 Assume that 
 \[
  f=\sum_{i=1}^t h_i f_{\uP_i, Q_i},
 \]
with $h_i \in \mathbb{Q}$. Then, by Theorem \ref{th:simpleStat},
\[
 M(f, n) = \sum_{i=1}^t h_i M( f_{\uP_i, Q_i}, n) = \sum_{i=1}^t h_i R_i (n) (n-k_i)!, 
\]
where $k_i$ is the length of $\underline{P_i}$ and the degree of $R_i(n)$ is no more than $deg(f_{\uP_i, Q_i})-d_i -c_i\leq m$, where $c_i=|\boldsymbol{C}(\underline{P_i})|$ and $d_i=|\boldsymbol{D}(\underline{P_i})|$. Combining the terms with the same $(n-k_i)!$ yields the equation 
\[
 M(f, n)= \sum_{j=0}^L U_j(n) (n-j)!, 
\]
where $U_j(n)$ is a polynomial of degree no more than $m$, and $L= \max(k_i) \leq \frac{m}{2}$.

As $\frac{(n-L+i)!}{(n-L)!}=(n-L+i)(n-L+(i-1))\cdots(n-L+1)$ is polynomial in $n$ of degree $i$, we obtain Equation \eqref{eq:statPolyForm} for $n \geq L$. In addition, $\frac{(n-L+i)!}{(n-L)!}$ for $0 \leq i\leq L+m$ forms a basis and Equation \eqref{T-form2} is obtained by expanding $U(n)$ under the basis $\{1, \frac{(n-L+1)!}{(n-L)!}, \frac{(n-L+2)!}{(n-L)!},\cdots, \frac{(n-L+L+m)!}{(n-L)!}\}$. 
\end{proof}

Theorem \ref{general-sum} allow us to obtain a closed form expression for $M(f,n)$ (and respectively for $\mathbb{E}(f)$), for any statistic $f$ whenever we know the exact values of $M(f,n)$ for a set of $L+m+1$ values of $n \geq L$, where $m=d(f)$. Then, we can take Equation \eqref{T-form2} and substitute each of these values for $n$. We get a system of $L+m+1$ linear equations, where the variables are the numbers $\alpha_{i}$, for $i\in [-L,m]$. After we solve it, we have a closed form expression for $M(f,n)$ as a linear combination of shifted factorials, coming from the same Equation \eqref{T-form2}. We used this approach and implemented a computer program, in order to obtain these closed forms for the aggregates of the statistics given as examples in Section \ref{sec:aggr}. Some of the results are listed below. \\

\textbf{Examples (formulas for aggregates of statistics):}\\

\begin{itemize}

\item[1.] $\cnt_{1324}$.\\
Recall that for the simple statistic $\cnt_{1324}=f_{\underline{P},Q}$, $\underline{P}=(1324,\emptyset,\emptyset)$ and $Q=1$. We have
\begin{equation*}
M(\cnt_{1324},n)=\frac{1}{24}n!-\frac{1}{6}(n+1)!+\frac{1}{8}(n+2)!-\frac{1}{36}(n+3)! + \frac{1}{576}(n+4)!.
\end{equation*}
In fact, a simple linearity of expectation argument gives that $M(\cnt_{P},n)= \frac{1}{k!}\binom{n}{k}n!$ for the number of occurrences of any classical pattern $P$ of length $k$. By using that the so-called Lah numbers, $L(k,j) = \binom{k-1}{j-1}\frac{k!}{j!}$, are the coefficients expressing rising factorials in terms of falling factorials, one can show that 
\begin{equation*} 
    M(\cnt_{P},n) = \frac{1}{k!}\binom{n}{k}n! =  \frac{(-1)^{k}}{k!}n! + \sum\limits_{j=1}^{k-1}\frac{(-1)^{k-j}}{(j!)^{2}(k-j)!}(n+j)! + \frac{1}{(k!)^{2}}(n+k)!.
\end{equation*}
Such a general formula can be derived for an arbitrary bivincular pattern.
\\
\item[2.] \emph{Descent drop}.\\
Recall that for the simple statistic $\drops=f_{\underline{P},Q}$, $\underline{P}=(21,\{1\},\emptyset)$ and $Q(s,w)=Q_{1}(s)Q_{2}(w)$, where $Q_{1}(s)=Q_{1}(t_{1},t_{2})=t_{2}-t_{1}$ and $Q_{2}(w)=1$. We have
\begin{equation*}
M(\drops,n)=-\frac{1}{2}(n+1)!+\frac{1}{6}(n+2)!.
\end{equation*}

\item[3.] \emph{Sum of peak squares}.\\
Recall that the statistic $\peakSqSum$ is a sum of the two simple statistics corresponding to the patterns $\underline{P}=(132,\{1,2\},\emptyset)$ and $\underline{P}=(231,\{1,2\},\emptyset)$, where the valuation polynomials for both statistics are $Q(s,w)=Q_{1}(s)Q_{2}(w)$ with $Q_{1}(s)=Q_{1}(t_{1},t_{2},t_{3})=t_{3}^{2}$ and $Q_{2}(w)=1$. We have
\begin{equation*}
M(\peakSqSum,n)=(n+1)!-\frac{5}{4}(n+2)!+\frac{1}{5}(n+3)!.
\end{equation*}

\end{itemize}

 \section{Higher moments of simple statistics}
 \label{sec:hMoments}
Our next goal is to show that the higher moments of statistics are also statistics. In order to investigate the higher moments, we will need to look at ordered tuples of occurrence of a given pattern. To do that, we will first define a \emph{merge} of two patterns, as done originally in \cite{partitions} for set partitions. In the definition given below, $g(S)\coloneqq \{g(x) \mid x\in S\}$, where $g$ is a function and $S$ is a set.

\begin{definition}[Merge of patterns]\label{def-merge}
Given are three patterns $\underline{P_1} = (x,\boldsymbol{C}(\underline{P_1}),\boldsymbol{D}(\underline{P_1}))$,\\ $\underline{P_2} = (y,\boldsymbol{C}(\underline{P_2}),\boldsymbol{D}(\underline{P_2}))$ and $\underline{P_3} = (z,\boldsymbol{C}(\underline{P_3}),\boldsymbol{D}(\underline{P_3}))$ of sizes $k_{1},k_{2}$ and $k_{3}$, respectively. A \emph{merge} of $\underline{P_1}$ and $\underline{P_2}$ onto $\underline{P_3}$ is a pair of increasing functions $m_{1}\mathbin{:}  [k_{1}]\to [k_{3}]$ and $m_{2}\mathbin{:} [k_{2}]\to [k_{3}]$, such that 
\begin{enumerate}
    \item $m_1([k_1])\cup m_2([k_2])= [k_3]$.
    \item for every $i,j\in[k_1]$, $(m_1(i),m_1(j))\in A(z)$ if and only if $(i,j)\in A(x)$ and for every $i,j\in[k_2]$, $(m_2(i),m_2(j))\in A(z)$ if and only if $(i,j)\in A(y)$.
    \item for every $j \in \boldsymbol{C}(\underline{P_1})$, $z^{-1}(m_{1}(x_{j+1})) = z^{-1}(m_{1}(x_{j}))+1$ and for every $j \in \boldsymbol{C}(\underline{P_2})$, $z^{-1}(m_{2}(y_{j+1})) = z^{-1}(m_{2}(y_{j}))+1$. In addition, 
   \[
    \boldsymbol{C}(\underline{P_3}) = \{z^{-1}(m_{1}(x_{j}))\mid j \in \boldsymbol{C}(\underline{P_1})\}\cup \{z^{-1}(m_{2}(y_{j}))\mid j \in \boldsymbol{C}(\underline{P_2})\}.
    \]
     \item for every $j \in \boldsymbol{D}(\underline{P_1})$, $m_{1}(j+1) = m_{1}(j)+1$ and for every $j \in \boldsymbol{D}(\underline{P_2})$, $m_2(j+1)=m_2(j)+1$. In addition,
     \[
     \boldsymbol{D}(\underline{P_3}) = \{m_1(j)\mid j\in \boldsymbol{D}(\underline{P_1})\} \cup \{m_2(j)\mid j\in \boldsymbol{D}(\underline{P_2}) \}.
     \]
\end{enumerate}
A merge will be denoted by $m_1, m_2\mathbin{:} \underline{P_1}, \underline{P_2}\rightarrow \underline{P_3}$.
\end{definition}

\begin{example} \leavevmode \newline
Let
$\underline{P_1}=(132,\{1\},\{2\})$,
$\underline{P_2}=(21,\emptyset,\emptyset)$ and
$\underline{P_3}=(2143,\{2\},\{3\})$.\\
Define the increasing functions $m_1$ and $m_2$ as follows: \\
$m_{1}(1) = 1$, $m_{1}(2) = 3$, $m_{1}(3) = 4$ \\
$m_{2}(1) = 1$, $m_{1}(2) = 2$
\end{example}

Note that for a merge, the pattern $\underline{P_{3}}$ is not uniquely defined by the functions $m_1$,$m_2$ and the patterns $\underline{P_1}$, $\underline{P_2}$. For instance, assume that $\underline{P_1}=321$,
$\underline{P_2}=21$ and 
$m_{1}(1) = 1$, $m_{1}(2) = 2$, $m_{1}(3) = 4$,
$m_{2}(1) = 3$, $m_{1}(2) = 4$. Then, $\underline{P_3}$ can be $4321$, $4231$ or $4213$. 

\begin{lemma}\label{lemma:higherMoments}
Let $\underline{P_1}$ and $\underline{P_2}$ be two patterns. For any $ \sigma \in S_n$, there is a one-to-one correspondence between the following sets.
 $$\{(s_1,s_2)\mathbin{:} s_1 \in_{\underline{P_{1}}} \sigma,s_2 \in_{\underline{P_{2}}} \sigma\}\leftrightarrow \{s_3 \in_{P_3} \sigma \mid m_1,m_2 \mathbin{:}\underline{P_1},\underline{P_2} \rightarrow \underline{P_3}\}$$.
\end{lemma} 
\vspace{-1cm}
\begin{proof}
Let $\underline{P_{1}}=(x,C(\underline{P_{1}}),D(\underline{P_{1}}))$ and $\underline{P_{2}}=(y,C(\underline{P_{2}}),D(\underline{P_{2}}))$.

$( \Longrightarrow )$ Assume that $s_{1}\in_{\underline{P_{1}}}$ and $s_{2}\in_{\underline{P_{2}}}$. Take the union of the elements of $s_1$ and $s_2$ and sort the elements of this union in increasing order. Let $s_{3}$ be the resulting increasing vector of numbers in $[n]$. As in the case of matchings and partitions, the maps $m_{a}$, for $a=1,2$, must be given by the unique function so that $m_{a}(i)=j$ if and only if the $i$-th smallest element of $s_{a}$ equals the $j$-th smallest element of $s_{3}$. If the elements of $s_{3}$ form the subsequence $\sigma_{i_{1}}\ldots \sigma_{i_{k_{3}}}$ in $\sigma$, then let $z = \red(\sigma_{i_{1}}\ldots \sigma_{i_{k_{3}}})$ and let $\underline{P_{3}} = (z,\boldsymbol{C}(\underline{P_{3}}),\boldsymbol{D}(\underline{P_{3}}))$, where $\boldsymbol{C}(\underline{P_3}) = \{z^{-1}(m_{1}(x_{j}))\mid j \in \boldsymbol{C}(\underline{P_1})\}\cup \{z^{-1}(m_{2}(y_{j}))\mid j \in \boldsymbol{C}(\underline{P_2})\}$ and $\boldsymbol{D}(\underline{P_3}) = \{m_1(j)\mid j\in \boldsymbol{D}(\underline{P_1})\} \cup \{m_2(j)\mid j\in \boldsymbol{D}(\underline{P_2}) \}$. 

We will show that $m_1,m_2 \mathbin{:}P_1,P_2 \rightarrow P_3$. One can easily verify that conditions (1) and (2) of Definition \ref{def-merge} hold. It remains to show that conditions (3) and (4) of the same definition also hold. We will do this just for $C(\underline{P_1})$ and $D(\underline{P_1})$ since one can proceed in the same way for $C(\underline{P_2})$ and $D(\underline{P_2})$. To check condition (3), it suffices to show that for every $j\in C(\underline{P_1})$, $z^{-1}(m_{1}(x_{j+1}))=z^{-1}(m_{1}(x_{j}))+1$. Indeed, the positions of the elements corresponding to $x_{j}$ and $x_{j+1}$ in every occurrence of $\underline{P_{1}}$, must be consecutive. Thus, since $s_{1}\in_{\underline{P_{1}}}\sigma$, the positions of $m_{1}(x_{j})$ and $m_{1}(x_{j+1})$ in $\sigma$, and consequently in $z$, must be consecutive, because $z$ is the reduction of $s_{3}$, which is the union of $s_1$ and $s_2$. Also, if $j\in D(\underline{P_1})$, then $t_{j+1}=t_{j}+1$, where $s_{1}=(t_{1},\ldots ,t_{k_{1}})$. Therefore, these two elements have consecutive values in $s_{3}$, as well, i.e., $m_1(j+1)=m_1(j)+1$. With that, we showed that $m_1,m_2 \mathbin{:}P_1,P_2 \rightarrow P_3$. Now, it is easy to check that $s_{3}\in_{\underline{P_{3}}}\sigma$.

$( \Longleftarrow )$ Let $s_{3}\in_{\underline{P_{3}}}\sigma$, where $s_{3}=(t_{1},t_{2}, \ldots , t_{k_{3}})$ is an increasing vector, $m_{1},m_{2}: \underline{P_{1}},\underline{P_{2}}\to \underline{P_{3}}$ and $\underline{P_{3}} = (z,\boldsymbol{C}(\underline{P_{3}}),\boldsymbol{D}(\underline{P_{3}}))$. Define $s_{1}\coloneqq t|_{m_{1},k_{1}}$ and $s_{2}\coloneqq t|_{m_{2},k_{2}}$, where $t|_{h,k} \coloneqq (t_{h(1)},t_{h(2)}, \ldots ,t_{h(k)})$. We must show that $t|_{m_{1},k_{1}}\in_{\underline{P_{1}}}\sigma$. One can similarly show that $t|_{m_{2},k_{2}}\in_{\underline{P_{2}}}\sigma$. Condition (2) of Definition \ref{def-merge} implies that the elements of $t|_{m_{1},k_{1}}$ are in the same relative order in $\sigma$ as the elements of $\underline{P_{1}}$. Now, assume that $j\in \boldsymbol{C}(\underline{P_{1}})$. We have to show that the positions of the elements $t_{m_{1}(x_{j})}$ and $t_{m_{1}(x_{j+1})}$ in $\sigma$ are consecutive. According to condition (3) of Definition \ref{def-merge}, $z(m_{1}(x_{j+1})) = z(m_{1}(x_{j}))+1$, i.e., $m_{1}(x_{j})$ and $m_{1}(x_{j+1})$ have consecutive positions in $z$ and $z^{-1}(m_{1}(x_{j}))\in  \boldsymbol{C}(\underline{P_3})$. Therefore, these positions must be also consecutive in $\sigma$ since $s_{3}\in_{\underline{P_{3}}}\sigma$. Finally, assume that $j\in \boldsymbol{D}(\underline{P_{1}})$. We have to show that $t_{m_{1}(j+1)}=t_{m_{1}(j)}+1$. According to condition (4) of Definition \ref{def-merge}, we must have that $m_{1}(j)\in \boldsymbol{D}(\underline{P_{3}})$ and $m_{1}(j+1)=m_{1}(j)+1$. Since $s_{3}\in_{\underline{P_{3}}}\sigma$, we have $t_{m_{1}(j)}+1 = t_{m_{1}(j)+1} = m_{1}(j+1)$.
\end{proof} 

Assume that $f$ is a simple statistic associated with the pattern $\underline{P_{1}}$ and valuation function $Q_{1}Q'_{1}$, whereas $g$ is a simple statistic associated with the pattern $\underline{P_{2}}$ and valuation function $Q_{2}Q'_{2}$.

Assume, also, that $m_1,m_2 \mathbin{:}\underline{P_1},\underline{P_2} \rightarrow \underline{P_3}$ for some $m_{1},m_{2}$ and $\underline{P_3}$. If $s_3=(t_1,t_2,\cdots, t_{k_3})\in_{\underline{P_{3}}}\sigma$ and $w_3=(\sigma^{-1}(t_1), \sigma^{-1}(t_2), \ldots, \sigma^{-1}(t_{k_3}))$, then let us define
\begin{align*}
 Q_{m_{1},m_{2},Q_{1},Q_{2}}(s_{3})\coloneqq Q_{1}(t|_{m_{1},k_{1}},n)Q_{2}(t|_{m_{2},k_{2}},n). 
   \end{align*}
   and 
   \begin{align*}
 Q'_{m_{1},m_{2},Q_{1},Q_{2}}(w_3)\coloneqq Q'_{1}(\sigma^{-1}(t|_{m_{1},k_{1}}),n)Q'_{2}(\sigma^{-1}(t|_{m_{2},k_{2}}),n).
   \end{align*}
   
\begin{theorem}\label{higher-moment} Let $\mathbbm{St}$ be the set of all permutation statistics thought
 of as functions $f \mathbin{:} \cup_{n}{S_n}\rightarrow \mathbb{Q}$.
 Then $\mathbbm{St}$ is closed under the operations of point-wise scaling, addition 
and multiplication. Thus, if $f$, $g$ $\in \mathbbm{St}$ and $a\in \mathbb{Q}$,
 then there exist permutation statistics $h_a$, $h_{+}$ and $h_{*}$ so that for all permutations $\sigma \in \mathbbm{St}$,
\begin{eqnarray*}
af(\sigma)&=&h_a(\sigma),\\
f(\sigma)+g(\sigma)&=&h_{+}(\sigma),\\
f(\sigma)g(\sigma)&=&h_{*}(\sigma).
\end{eqnarray*}
Furthermore, we have the following inequalities for the degrees:
$d(h_a)\leq d(f)$, $d(h_{+})\leq \max \{ d(f), d(g)\}$ and $d(h_{*})\leq d(f)+d(g)$.
\end{theorem}
\begin{proof}
The addition of two statistics is obviously a statistic by definition. Now, one can easily see that it suffices to show the existence of $h_{a}$ and $h_{*}$, when $f$ and $g$ are simple statistics. If $f$ corresponds to the pattern $\underline{P}$ and the valuation function $Q(s,w)=Q_1(s)Q_2(w)$, then let $h_{a}$ be the simple statistic corresponding to the same pattern $\underline{P}$ and valuation function $Q'(s,w)=aQ_1(s)Q_2(w)=Q'_1(s)Q_2(w)$. Clearly, $h_{a}$ is a statistic. To establish the fact that the product of two simple statistics is a statistic, we need Lemma \ref{lemma:higherMoments}. Let $f$ and $g$ have associated patterns $\underline{{P}_1}$, $\underline{{P}_2}$ and valuations functions $Q_1Q_1'$ and $Q_2Q_2'$, respectively. For any positive integer $n$, let $\sigma \in S_n$ and consider 
\begin{align*}
f_{\underline{{P}_1}, Q_1}(\sigma)g_{\underline{{P}_2}, Q_2}(\sigma) = \sum_{s_1\in _{\underline{{P}_1}} \sigma}Q_1(s_1) Q_1'(\sigma^{-1}(s_1))\sum_{s_2\in _{\underline{{P}_2}} \sigma}Q_2(s_2)Q_2'(\sigma^{-1}(s_2))\\
\overset{\text{\scalebox{0.5}{(by Lemma \ref{lemma:higherMoments})}}}= \sum_{\underline{P_{3}}}\left(\sum_{s_3\in_{\underline{P_3}}\sigma} \left(\sum_{m_1, m_2\mathbin{:}\underline{P_1}, \underline{P_2}\to \underline{P_3}} Q_{m_1,m_2,Q_1,Q_2}(s_3)Q_{m_1,m_2,Q_1,Q_2}'(\sigma^{-1}(s_3))\right)\right)
= \sum_{\underline{P_{3}}}f_{\underline{P_{3}},\tilde{Q}},
\end{align*}
where
$$\tilde{Q}(s_3) = \sum_{m_1, m_2\mathbin{:}\underline{P_1}, \underline{P_2}\to \underline{P_3}} Q_{m_1,m_2,Q_1,Q_2}(s_3)Q_{m_1,m_2,Q_1,Q_2}'(\sigma^{-1}(s_3))$$ 
for the fixed $\underline{P_{1}},\underline{P_{2}}$ and $\underline{P_{3}}$. We get that the product $fg$ is a finite sum of statistics and thus, it is a statistic itself. Indeed, this sum is finite since the number of patterns $\underline{P_{3}}$ that one can get as a merge of $\underline{P_{1}}$ and $\underline{P_{2}}$ is finite. Note that the bounds on the degrees of the statistics $h_{a}$, $h_{+}$ and $h_{*}$ follow directly from our proof and the definitions.
\end{proof}   
We will also need a generalization of Definition \ref{def-merge}. Let $\underline{P_1}$,$\underline{P_2}$, $\ldots$, $\underline{P_l}$ be $l$ patterns, where $k$ is the length of the pattern $\underline{P}$ and for each $i \in [l]$, $k_i$ is the length of the pattern $\underline{P_i}$. If we have the increasing functions $m_1\mathbin{:}[k_1]\to [k]$, $m_2\mathbin{:}[k_2]\to [k]$, $\ldots$ , $m_l\mathbin{:}[k_l]\to [k]$, then a \emph{merge} of  these $l$ patterns corresponding to the listed functions is denoted by $m_1,m_2, \ldots, m_l\mathbin{:} \underline{P_1}, \underline{P_2},\ldots,\underline{P_l}\to \underline{P}$ or by the shorthand $\mathcal{M}_{l}\mathbin{:}\Pi_{l}\to\underline{P}$. Similarly, for any $\sigma \in S_n$ one can establish an analogue of Lemma \ref{lemma:higherMoments}. We state this result without a proof. 
\begin{lemma}\label{r-moment version}
Assume that we have the $r$ patterns $(\underline{P_1},\boldsymbol{C}(\underline{P_1}), \boldsymbol{D}(\underline{P_1}))$, $(\underline{P_2},\boldsymbol{C}(\underline{P_2}), \boldsymbol{D}(\underline{P_2}))$, $\ldots$ , $(\underline{P_{r}}, \boldsymbol{C}(\underline{P_r}), \boldsymbol{D}(\underline{P_r}))$. There is a one-to-one correspondence between the following sets.
\begin{align*}
\{(s_1,s_2, \ldots, s_r)\mid s_1\in_{\underline{P_1}}\sigma, s_2\in_{\underline{P_2}}\sigma, \ldots, s_r\in_{\underline{P_r}}\sigma\}\\
\leftrightarrow \{s\in_{\underline{P}} \sigma \mid m_1,m_2,\ldots,m_r\mathbin{:} \underline{P_1},\underline, \ldots, \underline{P_r}\rightarrow P\}. 
\end{align*}
\end{lemma}
Using this lemma, one can obtain analogously that the product of $r$ statistics of degrees $d_{1}, \ldots , d_{r}$ is a statistic of degree not more than $\sum_{j=1}^{r}d_{j}$. We use this observation to obtain the following result.

\begin{theorem}
\label{th:main}
Let $f$ be any statistic of degree $m$. Then, for any positive integer $r$, the $r$-th moment of $f$ is given by 
\begin{equation}\label{eq-higher_moment}
M(f^r,n)= \sum_{-I\leq i\leq J} \alpha_i(n+i)!,
\end{equation} 
where $I$ and $J$ are constants that satisfy $-I\geq \frac{-rm}{2}$, $J\leq mr$ and $n\geq I$, and the $\alpha _i$'s are rational constants.
\end{theorem}
\begin{proof}
Let $f=\sum_{i=0}^{t}\beta_i f_{\underline{P_i},Q_i}$. We have 
\begin{equation}
\label{eq:rMoments}
\resizebox{0.94\hsize}{!}{
$\begin{aligned}
    M(f^r,n) = \sum_{\sigma\in S_n}{\left(\sum_{i=1}^t \beta_i f_{\underline{P_i},Q_i}(\sigma)\right)}^r =
    \sum_{\sigma\in S_n}\sum_{\underline{P^{r}}}\gamma_{j}\left(\sum_{s\in_{\underline{P^r}}\sigma}\left(\sum_{\mathcal{M}_{r}\mathbin{:}\Pi_{r}\to \underline{P^r}}\prod_{i=1}^{r}Q_{i}(t\mid_{m_{i},k_{i}},\sigma^{-1}(t\mid_{m_{i},k_{i}}))\right)\right)\\
    \overset{\text{\scalebox{0.55}{(by Lemma \ref{r-moment version})}}}=\sum_{\sigma\in S_n}\sum_{\underline{P^{r}}}\gamma_{j}f_{\underline{P^{r}},\tilde{Q}}(\sigma) = \sum_{\underline{P^{r}}}\gamma_{j}M(f_{\underline{P^{r}},\tilde{Q}},n),
\end{aligned}$
}
\end{equation}
for some constants $\gamma_{j}\in \mathbb{Q}$.
Each of the statistics $f_{\underline{P^{r}},\tilde{Q}}$ is a summation of products of $r$ statistics, with each of them being of degree not more than $m$. Thus, $f_{\underline{P^{r}},\tilde{Q}}$ is a statistic of degree not more than $rm$, for every $\underline{P^{r}}$. Therefore, by Theorem \ref{general-sum}, we get
\begin{equation}
\label{eq:linearFormMain}
M(f^r,n)= \sum_{-L\leq i\leq rm} \alpha_i(n+i)!  
\end{equation}
where $L\leq \frac{m}{2}$.
\end{proof}
In order to establish Lemma \ref{lemma:mainVinc}, which is an important special case of Theorem \ref{th:main}, we will need the lemma below.
\begin{lemma}
\label{lemma:boundMerge}
Consider a merge of the vincular patterns $\underline{P_{1}} = (x,\boldsymbol{C}(\underline{P_{1}}))$ and $\underline{P_{2}} = (y,\boldsymbol{C}(\underline{P_{2}}))$ onto $\underline{P_{3}} = (z,\boldsymbol{C}(\underline{P_{3}}))$, where $x$, $y$ and $z$ are of lengths $k_1$, $k_2$ and $k_3$, respectively and the values of $|\boldsymbol{C}(\underline{P_{1}})|$, $|\boldsymbol{C}(\underline{P_{2}})|$ and $|\boldsymbol{C}(\underline{P_{3}})|$ are $c_1$, $c_2$ and $c_3$, respectively. Then, 
\[
k_{3}-c_{3}\leq (k_{1}+k_{2})-(c_{1}+c_{2}).
\]
\end{lemma}
\begin{proof}
Part (3) of Definition \ref{def-merge} allows us to write the following:
\begin{equation*}
\resizebox{0.93\hsize}{!}{
$\begin{aligned}
k_{3}-c_{3} = (k_{1}+k_{2}-(m_{1}([k_{1}])\cap m_{2}([k_{2}])))-(c_{1}+c_{2}-(\{m_{1}(x_{i})\mid i\in \boldsymbol{C}(\underline{P_{1}})\}\cap\{m_{2}(y_{j})\mid j\in \boldsymbol{C}(\underline{P_{2}})\})) = \\
(k_{1}+k_{2})-(c_{1}+c_{2}) - [(m_{1}([k_{1}])\cap m_{2}([k_{2}])) - (\{m_{1}(x_{i})\mid i\in \boldsymbol{C}(\underline{P_{1}})\}\cap\{m_{2}(y_{j})\mid j\in \boldsymbol{C}(\underline{P_{2}})\})].    
\end{aligned}
$}
\end{equation*}
Thus, it suffices to show that
\begin{equation*}
    (m_{1}([k_{1}])\cap m_{2}([k_{2}])) - (\{m_{1}(x_{i})\mid i\in \boldsymbol{C}(\underline{P_{1}})\}\cap\{m_{2}(y_{j})\mid j\in \boldsymbol{C}(\underline{P_{2}})\})\geq 0,
\end{equation*}
but the latter is clearly true since $\boldsymbol{C}(\underline{P_{1}})$ and $\boldsymbol{C}(\underline{P_{2}})$ are subsets of $[k_{1}]$ and $[k_{2}]$, respectively.
\end{proof}

\begin{theorem}
\label{lemma:mainVinc}
If $\underline{P}$ is a vincular pattern of length $k$, such that $|\boldsymbol{C}(\underline{P})|=c$, then 
\begin{equation}
\label{eq:linFormVinc}
M(\cnt_{\underline{P}}^r,n)= \sum_{0\leq i\leq r(k-c)} \alpha_i(n+i)!,
\end{equation}
 for $n\geq rk$.
\end{theorem}
\begin{proof}
One can easily prove the following equality (Lemma \ref{lemma:ExpBivinc}, proved in the next section, gives a generalisation):
\[
M(\cnt_{\underline{P}},n) =\frac{\binom{n-c}{k-c}}{k!}n!.
\]
Since $\binom{n-c}{k-c}$ is a polynomial in $n$ of degree $k-c$, the statement of the lemma holds, when $r=1$. For bigger values of $r$, we can look at Equation \eqref{eq:rMoments} and plug in $t=1$, $\beta_{1}=1$ and $Q=1$ for all valuation functions $Q$, as well as $\underline{P_{1}}=\underline{P_{2}} = \ldots \underline{P_{r}} = \underline{P}$ . We will get that 
\begin{equation}
\label{eq:tmp}
M(\cnt_{\underline{P}}^{r},n) = \sum_{\underline{P^r}}\delta_{j}M(\cnt_{\underline{P^r}},n),
\end{equation}
where the summation is over all possible merges $\underline{P^r}$ of $r$ copies of $\underline{P}$ and where $\delta_j$ are some rational constants. Using Lemma \ref{lemma:boundMerge}, we can see that each of the patterns $\underline{P^r}=(z,\boldsymbol{C}(\underline{P^r}))$ is a vincular pattern with $|z|-|\boldsymbol{C}(\underline{P^r})|\leq r(k-c)$. Therefore, each of the aggregates $M(\cnt_{\underline{P^r}},n)$ can be written in the form, as in the right side of Equation \eqref{eq:linFormVinc}. After we substitute these forms in the right side of Equation \eqref{eq:tmp} and regroup, we see that the claim holds.
\end{proof}

Theorem \ref{th:main} and Theorem \ref{lemma:mainVinc} generalize a result of Zeilberger \cite[Main formula]{Z.I}. What he proved is that for any classical pattern $\underline{P}$ of length $k$, $\mathbb{E}(\cnt_{\underline{P}}^{r})$ is a polynomial of degree $rk$. In the same article, he used this observation to get the polynomials for the second and the third moments of the statistic $\cnt_{\underline{P}}$, for various classical patterns $\underline{P}$. To do that, he implemented a computer program that fits the actual values of this statistic for $0,1,\ldots , rk$ to a polynomial of degree $rk$. Below, we give explicit expressions for the second moment of some of the statistics introduced in Section \ref{sec:defs}. We use the same approach by fitting small values of these statistics to the right side of Theorem \ref{th:main} or Theorem \ref{lemma:mainVinc}, in order to find the coefficients $\alpha_{i}$.\\

\textbf{Examples (formulas for aggregates of higher moments):}
\begin{itemize}
\item[1.] Second moment of the double ascents.\\
\begin{equation*}
M(\cnt_{\underline{123}}^2,n)=-\frac{1}{12}n!-\frac{1}{15}(n+1)!+\frac{1}{36}(n+2)!.
\end{equation*}

\item[2.] Second moment of $\cnt_{\scaleto{\Vectorstack{1\overline{23} 3\underline{12}}}{16pt}}$.\\
\begin{equation*}
M(\cnt_{\scaleto{\Vectorstack{1\overline{23} 3\underline{12}}}{16pt}}^2,n)=\frac{1}{2}n!-\frac{9}{28}(n+1)!+\frac{29}{672}(n+2)!+\frac{11}{10080}(n+3)!-\frac{1}{45360}(n+4)!.
\end{equation*}
\end{itemize}

Several important simple statistics have unit valuation function associated to them, i.e., $Q(s,w)=1$. For these cases, we give the following important corollary from Theorem \ref{th:main}, which is an analogue of \cite[Proposition 3.5]{matchings} and will be substantially used in the next two sections. 
\begin{corollary}
\label{corr:main}
Let $\underline{P}$ be a pattern of length $k$ with $|\boldsymbol{C}(\underline{P})|=c$, $|\boldsymbol{D}(\underline{P})|=d$ and unit valuation function. Then,
\begin{equation}
\label{eq:corr}
M(\cnt_{\underline{P}}^{r},n) = \sum_{\tilde{k},\tilde{c},\tilde{d}}w^{(r)}_{\tilde{k},\tilde{c},\tilde{d}}\binom{n-\tilde{c}}{\tilde{k}-\tilde{c}}\binom{n-\tilde{d}}{\tilde{k}-\tilde{d}}(n-k)!,
\end{equation}
where $w^{(r)}_{\tilde{k},\tilde{c},\tilde{d}}$ is the number of ways to merge $r$ copies of $\underline{P}$ and get a pattern $\underline{P^{r}}$ of length $\tilde{k}$, with $|C(\underline{P^{r}})|=\tilde{c}$, $|D(\underline{P^{r}})|=\tilde{d}$ and where $k\leq \tilde{k}\leq rk$, $c\leq \tilde{c}\leq rc$ and $d\leq \tilde{d}\leq rd$.
\end{corollary}
\begin{proof}
Take Equation \eqref{eq:rMoments} in the proof of Theorem \ref{th:main} and plug in $t=1$, $\beta_{1}=1$, $Q=1$ for all valuation functions $Q$ and $\underline{P_{1}}=\underline{P_{2}} = \ldots \underline{P_{r}} = \underline{P}$. Then, $M(\cnt_{\underline{P}}^{r},n) = \sum_{\underline{P^r}}\gamma_{j}M(f_{\underline{P^r},\tilde{Q}},n)=\sum_{\underline{P^r}}\gamma_{j}'M(\cnt_{\underline{P^r}},n)$, for some rational constants $\gamma_{j}'$. In addition, use Lemma \ref{lemma:ExpBivinc} to get that
\[
M(\cnt_{\underline{P}},n) =\frac{\binom{n-\tilde{c}}{\tilde{k}-\tilde{c}}\binom{n-\tilde{d}}{\tilde{k}-\tilde{d}}}{n_{(\tilde{k})}}n! = \binom{n-\tilde{c}}{\tilde{k}-\tilde{c}}\binom{n-\tilde{d}}{\tilde{k}-\tilde{d}}(n-\tilde{k})!,
\]
for every pattern $\underline{P}$ of length $\tilde{k}$, with $|C(\underline{P})|=\tilde{c}$ and $|D(\underline{P})|=\tilde{d}$.
\end{proof}

\section{Descents and minimal descents. Explicit formulas for the higher moments.}
\label{sec:descents}
The results from the previous section can be used to obtain an explicit formula for the $r$-th moment of some permutation statistics. In this section, we illustrate how this can be done for the descents and the minimal descents statistics. We will use the following simple lemma.
\begin{lemma}
\label{lemma:ExpBivinc}
For any bivincular pattern $\underline{P}$ of length $k$, such that $|\boldsymbol{C}(\underline{P})|=c$ and $|\boldsymbol{D}(\underline{P})|=d$,
\[
\mathbb{E}(\cnt_{\underline{P}},n) =\frac{\binom{n-c}{k-c}\binom{n-d}{k-d}}{n_{(k)}}.
\]
\end{lemma}
\begin{proof}
Let $I$ be the set of possible positions for an occurrence of $\underline{P}$ in a permutation of length $n$. Similarly, let $J$ be the set of possible values of the numbers in such an occurrence. By linearity of expectation, we have that
\[
\mathbb{E}(\cnt_{\underline{P}},n) = \sum_{i\in I,j\in J} X_{i,j},
\]
where the random variable $X_{i,j}\coloneqq 1,$ if the set of possible values with index $j$ are at the set of possible positions with index $i$, and these values are in the relative order determined by the permutation $P$. Otherwise, $X_{i,j}\coloneqq 0$. Note that when we choose a permutation of length $n$ at random, $\mathbb{E}(X_{i,j})=\frac{1}{n_{(k)}}$. Also, note that $|I|=\binom{n-c}{k-c}$ and $|J|=\binom{n-d}{k-d}$.
\end{proof}

Consider the statistic $\des = \cnt_{\underline{21}}$. It is well known that the number of permutations of length $n$ having $k$ descents is given by the \emph{Eulerian numbers} and the corresponding distribution is called \emph{Eulerian distribution}. A comprehensive source dedicated to Eulerian numbers is the book \cite{petersen}. Its preface and the notes at the end of Chapter 1 provide a good historical overview. A recent article by Hwang et al. gives a complicated recurrence relation as a way to calculate the higher moments of the Eulerian distribution and a family of other distributions with generating functions satisfying a similar relation (see \cite[Section 2.2]{hwang}). Below, we give a direct summation formula for the $r$-th moment of the Eulerian distribution.

\begin{theorem}
\label{th:eulerMoments}
Consider a random permutation of length $n$ and $r\geq 1$. Then,
\begin{equation*}
    \Scale[0.95]{\mathbb{E}(\des^{r}) = \sum\limits_{m=2}^{min(n,2r)}\sum\limits_{u=1}^{\lfloor \frac{m}{2}\rfloor}\left(\sum\limits_{w=0}^{m-u}(-1)^{w}\binom{m-u}{w}(m-u-w)^{r}\right)\left(\sum\limits_{\substack{q_{1}+\dots + q_{u}=m\\q_{i}\geq 2}}\binom{m}{q_{1},\dots ,q_{u}}\right)\frac{\binom{n-(m-u)}{u}}{m!}}.
\end{equation*}
\end{theorem}
\begin{proof}
Use Corollary \ref{corr:main} and note that for $\underline{P}=\underline{21}$, $d=0$ and $c=1$. 
Let us find the numbers $w^{(r)}_{\tilde{k},\tilde{c}}$ for the pattern $\underline{21}$. We will need to sum over all possible merges $\underline{P^{r}}$ depending on their length $\tilde{k}$ and the value $\tilde{c}$ of $|\boldsymbol{C}(\underline{P^{r}})|$. Instead of $\tilde{k}$, we will write $m$. Any of the patterns $\underline{P^{r}}$ can have between $m=2$ and $m=2r$ letters. For a fixed $m$, any such pattern can be comprised of $u$ segments of consecutive letters, where $1\leq u \leq \lceil \frac{m}{2}\rceil$. For example, $q=\underline{43}\, \underline{61}\, \underline{752}$ has length $m=7$ and is comprised of three segments of consecutive letters, namely $43$, $61$ and $752$. Note that getting a pattern $\underline{P}$ with $u$ segments requires merging at least $m-u$ copies of the pattern $\underline{21}$ since a segment of length $h$ requires merging at least $h-1$ copies of $\underline{21}$. For instance, the segment $\underline{752}$ in the pattern $q$ above can be obtained after merging multiple copies of $\underline{21}$, corresponding either to $\underline{75}$ or to $\underline{52}$ and at least one copy corresponding to each of them. In general, for a merge with $u$ segments, each of the $r$ copies of the descent pattern \underline{21} must correspond to one out of $m-u$ pairs of consecutive elements and we must have at least one copy for each of these pairs. The inclusion-exclusion principle gives us $\sum\limits_{w=0}^{m-u}(-1)^{w}\binom{m-u}{w}(m-u-w)^{r}$ ways to achieve that. In addition, every segment must be a decreasing sequence of elements. If the lengths of the segments in the pattern $\underline{P^{r}}$ are denoted by $q_{1},\ldots ,q_{u}$, then we must have $q_{1}+ \cdots + q_{u} = m$ and $q_{i}\geq 2$ for each $1\leq i\leq u$. Thus, for every such composition of $m$, we can choose the numbers in each of the segments in $\binom{m}{q_{1},\dots ,q_{u}}$ ways. Finally, for every pattern $\underline{P^{r}}$ with $u$ segments, $|\boldsymbol{C}(\underline{P^{r}})| = m-u$. Therefore $w^{(r)}_{\tilde{k},\tilde{c}}=w^{(r)}_{m,m-u} = \sum\limits_{u=1}^{\lfloor \frac{m}{2}\rfloor}\left(\sum\limits_{w=0}^{m-u}(-1)^{w}\binom{m-u}{w}(m-u-w)^{r})\right)\sum\limits_{\substack{q_{1}+\dots + q_{u}=m\\q_{i}\geq 2}}\binom{m}{q_{1},\dots ,q_{u}}$ and $M(\cnt_{\underline{P^{r}}},n) = \frac{\binom{n-(m-u)}{u}}{m!}n!$, by Lemma \ref{lemma:ExpBivinc}. Our goal is to find $E(\des^{r})$, so we are dividing both sides by $n!$ to obtain the desired formula.
\end{proof}

Similarly, we can obtain the moments of the minimal descents statistic $\cnt_{\scaleto{\Vectorstack{\overline{12} \underline{21}}}{16pt}}$, i.e., a descent, such that the two numbers in it are consecutive. In the literature, this statistic is also known as \emph{adjacency} and we will denote it by $\adj$. The following Theorem will be used in Section \ref{subsec:bivinc}.

\begin{theorem}
\label{th:adj}
Consider a random permutation of length $n$ and $r\geq 1$. Then,
\begin{equation*}
    \Scale[1]{\mathbb{E}(\adj^{r}) = \sum\limits_{m=2}^{min(n,2r)}\sum\limits_{u=1}^{\lfloor \frac{m}{2}\rfloor}\left(\left(\sum\limits_{w=0}^{m-u}(-1)^{w}\binom{m-u}{w}(m-u-w)^{r}\right)\binom{m-u-1}{u-1}u!\frac{\binom{n-(m-u)}{u}^{2}}{n_{(m)}}\right)}.
\end{equation*}
\end{theorem}
\begin{proof}
Proceed as in the proof of the previous Theorem \ref{th:eulerMoments}. One difference is that now, for a pattern $\underline{P^r}$ of length $m$ with $u$ segments, the values of the numbers in each segment must be consecutive. Thus, instead of $\sum\limits_{\substack{q_{1}+\cdots + q_{u}=m\\q_{i}\geq 2}}\binom{m}{q_{1},\ldots ,q_{u}}$ possible ways to determine the numbers in a pattern with $u$ segments, we have just $u!$ such segments for every solution of $q_{1}+\cdots + q_{u}=m$, where $q_{i}\geq 2$. By using the stars and bars model, one can see that the number of these solutions is exactly $\binom{m-u-1}{u-1}$. In addition, one can see that $|\boldsymbol{C}(\underline{P^r})|=|\boldsymbol{D}(\underline{P^r})|=m-u$, for every pattern $\underline{P^r}$ with $u$ segments and therefore by Lemma \ref{lemma:ExpBivinc}, we get $M(\cnt_{\underline{P^{r}}},n) =\frac{\binom{n-(m-u)}{u}^{2}}{n_{(m)}}$.
\end{proof}

\section{Central limit theorems for $\cnt_{\underline{P}}$.}
\label{sec:CLTs}
The normal distribution is frequently appearing in the context of combinatorial enumeration \cite[Chapter 3]{HEC}. A major reason is, of course, the central limit theorem, which gives us that under rather general circumstances, when independent random variables are added, their properly normalized sum converges in distribution to the normal distribution. Formally, a random variable $X$ is \emph{normally distributed} when 
\[
\mathbb{P}(X\leq x) = \frac{1}{\sqrt{2\pi}}\int_{-\infty}^{x}e^{-t^{2}/2}dt.
\]
In this section, we will reprove some limiting laws for the random variable $\cnt_{\underline{P}}$, which counts the number of occurrences of the pattern $\underline{P}$ in a given permutation. 

\subsection{Classical Patterns}
\label{sec:classical}
Recall that if $\boldsymbol{C}(\underline{P})$ and $\boldsymbol{D}(\underline{P})$ are empty, then $\underline{P}$ is a classical pattern. The limiting normality of $\cnt_{\underline{P}}$, when $\underline{P}$ is a classical pattern was first established by B\'{o}na \cite{bona}. He uses the method of \emph{dependency graphs} and the \emph{Janson dependency criterion}. This method is used when we have a set of partially dependent random variables, for every value of $n$, and we want to prove that the sum of these variables has a certain asymptotic distribution. To obtain a dependency graph for a set of random variables, we take a vertex for each variable and connect the dependent random variables by edges. We can construct a dependency graph for each value of $n$. The idea of the method is that if the degrees of the vertices in the obtained sequence of dependency graphs do not grow too fast, then the corresponding variables behave as if independent and their sum is asymptotically normal \cite{ferayDepGraphs}. Janson's criterion gives one sufficient condition for this asymptotic normality, quantifying that the degrees do not grow too quickly. A main fact that B\'{o}na uses when checking the criterion is a lower bound on the variance of $\cnt_{\underline{P}}$. In this subsection, we reprove this result by using Corollary \ref{corr:main} and Lemma \ref{lemma:classical} given below, which was established by Burstein and H\"{a}st\"{o} \cite{burstein}. This gives a new proof that $\cnt_{\underline{P}}$ has asymptotically normal distribution. We also provide a new interpretation of Lemma \ref{lemma:classical}.

Let $A_{\sigma}(r)$ denotes the set of possible merges of two copies of the pattern $\underline{\sigma}$, which is of length $k$, and where the resulting pattern is of length $r$. Formally, $A_{\sigma}(r)$ can be defined as the set of triples $(\pi, m_{1}, m_{2})$, such that $m_1,m_2 \mathbin{:}\underline{\sigma} ,\underline{\sigma} \rightarrow \underline{\pi}$ and $\pi\in S_{r}$. However, it will be more convenient for us to look at the subsequences of $\pi$ formed by the images of the functions $m_1$ and $m_2$, i.e., we will use the following equivalent definition.
\begin{definition}
For $\sigma\in S_{k}$, let 
\[
A_{\sigma}(r) \coloneqq \{(\underline{\pi}, x,y)\mid \pi\in S_{r},\text{ }x,y\in \subs(\pi),\text{ }\red(x) {=} \sigma,\text{ }\red(y){=}\sigma,\text{ } |x\cap y|=2r-k\},
\]
where $\subs(\pi)$ denotes the set of the subsequences of the permutation $\pi$.
\end{definition}

For instance, if $\underline{\sigma}$ is the classical pattern $312$, then $A_{312}(5)$ contains $(54213, 523, 413)$, since $\red(523)=312$, $\red(413)=312$ and these two subsequences have exactly one common element (see Figure 1). 
\begin{table}
    $
\begin{array}{|c|c|c|c|c|c|}
\hline
5 & 4 & 2 & 1 & 3\\ \hhline{|=|=|=|=|=|}
5 &   & 2 &   & 3\\ \hline
  & 4 &   & 1 & 3\\ \hline
\end{array}$
    \caption{Merge of two copies of the pattern $312$.}
    \end{table}
Let $a_{\sigma}(r)\coloneqq |A_{\sigma}(r)|$.
\begin{lemma}[{Burstein and H\"{a}st\"{o}, \cite[Lemma 4.3]{burstein}}]
\label{lemma:classical}
For any classical pattern $\sigma = \sigma_{1}\cdots\sigma_{k}$, 
\begin{equation}
\label{eq:burnstein}
a_{\sigma}(2k-1) > \binom{2k-1}{k}^{2}.    
\end{equation}
\end{lemma}
\begin{example}
$k=2$ and $\sigma = 21$. Then, $\binom{2k-1}{k}^2 = 9$ and $a_{\sigma}(2k-1)=a_{21}(3)=10$, since $A_{21}(3)$ consists of the $10$ triples $(\pi, x,y)$ given below:

$\pi = 321$: $(321,32,31)$, $(321,31,32)$, $(321,32,21)$, $(321,21,32)$, $(321,31,21)$, $(321,21,31)$.

$\pi = 312$: $(312,31,32)$, $(312,32,31)$.

$\pi = 231$: $(231,21,31)$, $(231,31,21)$. 
\end{example}

Now, we are ready to prove the bound for the variance of $\cnt_{\underline{P}}$ used by B\'{o}na.
\begin{theorem}
\label{th:boundBona}
Let $X_{n}\coloneqq \cnt_{\sigma}$ be the number of occurrences of a classical pattern $\sigma = \sigma_{1}\cdots\sigma_{k}$ in a random permutation of length $n$. Then, there exists $c>0$, such that for all $n$,
\[
\Var(X_{n})\geq cn^{2k-1}.
\]
\end{theorem}
\begin{proof}
Since $\sigma$ is a classical pattern, Lemma \ref{lemma:ExpBivinc} gives us that $\mathbb{E}(X_{n})=\frac{\binom{n}{k}}{k!}$. Using this fact and Corollary \ref{corr:main}, we obtain
\begin{equation*}
\begin{split}
\Scale[0.95]{\Var(X_{n}) = \mathbb{E}(X_{n}^{2}) - \mathbb{E}^{2}(X_{n}) = [a_{\sigma}(2k)\frac{\binom{n}{2k}}{(2k)!} + a_{\sigma}(2k-1)\frac{\binom{n}{2k-1}}{(2k-1)!}+ \mathcal{O}(n^{2k-2})] - \frac{\binom{n}{k}^{2}}{(k!)^{2}}}.    
\end{split}
\end{equation*}
We know that $\binom{n}{k} = \frac{(n)_{k}}{k!}$ and that $(n)_{k} = \sum\limits_{i=0}^{k}s(k,i)n^{i}$, where $s(k,i)$ are the \emph{Stirling numbers of the first kind}. We have $s(k,i) = (-1)^{k-i}{ k\brack i}$, where ${ k\brack i}$ is the number of permutations in $S_{k}$ with $i$ disjoint cycles. In particular, ${ k\brack k-1} = \binom{k}{2}$. Therefore,
\begin{equation*}
\begin{split}
\Scale[0.95]{\Var(X_{n}) = [a_{\sigma}(2k)\frac{n^{2k}-\binom{2k}{2}n^{2k-1}}{((2k)!)^{2}} + a_{\sigma}(2k-1)\frac{n^{2k-1})}{((2k-1)!)^{2}}] - \frac{n^{2k}-2\binom{k}{2}n^{2k-1}}{(k!)^{4}}+ \mathcal{O}(n^{2k-2}).}
\end{split}
\end{equation*}
It is easy to see that $a_{\sigma}(2k) = \binom{2k}{k}^{2}$ since a merge of length $2k$ of two copies of $\underline{\sigma}$ is uniquely determined by the set of $k$ positions among $[2k]$, where the first copy will be placed and the set of $k$ values among $[2k]$ at these positions. The values and the positions for the letters of the second copy are those remaining. Then, one can see that the coefficient of $\Var(X_{n})$ in front of $n^{2k}$ is $0$ and the coefficient in front of $n^{2k-1}$ is
\begin{equation*}
    \frac{-\binom{2k}{k}^{2}\binom{2k}{2}}{((2k)!)^{2}} + \frac{a_{\sigma}(2k-1)}{((2k-1)!)^{2}} + \frac{2\binom{k}{2}}{(k!)^{4}}.
\end{equation*}
Simplify the last expression to get that this coefficient is positive, only if
\begin{equation*}
    a_{\sigma}(2k-1)> \binom{2k-1}{k}^{2},
\end{equation*}
which follows from Lemma \ref{lemma:classical}.
\end{proof}
It is interesting to note that Burstein and H\"{a}st\"{o} obtained the same bound for the variance of $X_{n}$ in \cite{burstein}, but they did not state that it implies the central limit theorem for $\cnt_{\underline{P}}$. At the same time, in \cite{bona}, B\'{o}na proved the bound independently and did not cite the work of Burstein and H\"{a}st\"{o}.  

The proof of Lemma \ref{lemma:classical}, found in \cite{burstein}, is algebraic. As a first step, this proof shows that $a_{\sigma}(2k-1)$ is the trace of a product of two symmetric matrices, for which we know that they have only positive eigenvalues. In addition, one of the eigenvalues of the product matrix turns out to be $\binom{2k-1}{k}^{2}$. The result follows, since the trace of a matrix equals the sum of its eigenvalues. Next, we give an interpretation of Lemma \ref{lemma:classical}, which may be useful to obtain a combinatorial proof for it. 

Let 
\[A_{\sigma, \sigma'}(r) \coloneqq \{(\pi, x,y)\mid \pi\in S_{r},\text{ }x,y\in \subs(\pi),\text{ }\red(x) {=} \sigma,\text{ }\red(y){=}\sigma',\text{ } |x\cap y|=2r-k\},
\]
be the set of merges of length $r$ for the permutations $\sigma\in S_k$ and $\sigma'\in S_k$, corresponding to the patterns $\underline{\sigma}$ and $\underline{\sigma'}$, respectively. Let $a_{\sigma , \sigma'}(r) = |A_{\sigma, \sigma'}(r)|$. 

\begin{theorem}
\label{th:interpretation}
Lemma \ref{lemma:classical} is equivalent to 

\begin{equation}
    a_{\sigma}(2k-1) > \mathbb{E}(a_{\sigma , \sigma'}(2k-1)),
\end{equation}
where $\sigma\in S_k$ is a fixed classical pattern and $\sigma'\in S_{k}$ is chosen uniformly at random.
\end{theorem}
\begin{proof}
First, note that $\binom{2k-1}{k}^2$, which is the right-hand side of Equation \eqref{eq:burnstein} in Lemma \ref{lemma:classical}, can be written as $\frac{\binom{2k-1}{k}}{k}\binom{2k-1}{k}k$. Then, observe that $\binom{2k-1}{k}k$ is the number of ways to choose the $k$ positions from $[2k-1]$ for the numbers of the subsequence $x$ (that is order isomorphic to $\sigma$), as well as the position of the common element $c$ for $x$ and the subsequence $y$ (that is order-isomorphic to $\sigma'$). For each of these choices, we can select the values of the numbers of $x$ at the already selected positions in $\binom{2k-1}{k}$ ways. Once this choice is made, the values of $x,y$ and $c$ are uniquely determined. Suppose that $c$ has to be at position $p$ in $y$. Since $\sigma'$ is chosen uniformly at random, we have probability $\frac{1}{k}$ for the element $c$ to be at position $p$ in $y$. This gives $\frac{\binom{2k-1}{k}}{k}$ for the expected number of merges when we know the positions of the elements of $x$ and the position of $c$. Therefore, 
\[
\mathbb{E}(a_{\sigma , \sigma'}(2k-1)) = \frac{\binom{2k-1}{k}}{k}\binom{2k-1}{k}k = \binom{2k-1}{k}^2.
\]
\end{proof}
Interestingly, if $\sigma$ is fixed, $a_{\sigma , \sigma'}(k,2k-1)$ does not necessarily reach its maximum when $\sigma'=\sigma$. For instance, $a_{1324 , 1234}(4,7)>a_{1324 , 1324}(4,7)$. However, since we know that Lemma \ref{lemma:classical} holds, Theorem \ref{th:interpretation} gives us that when $\sigma'=\sigma$, we always get a value greater than the expectation over $\sigma'$. 

\subsection{Vincular Patterns}
\label{sec:vincular}
Recall that if $\boldsymbol{D}(\underline{P})$ is empty, then $\underline{P}$ is a vincular pattern and to denote it, we write $P$ with the positions $i$ and $i+1$ of $P$ underlined, for every $i\in\boldsymbol{C}(\underline{P})$. The \emph{blocks} of a vincular pattern are the groups of numbers at consecutive positions at the pattern, such that their corresponding numbers in an occurrence must be at consecutive positions, as well. For example, if $\underline{P} = (135246, \{1,2,5\},\emptyset)=\underline{135}2\underline{46}$, then $\boldsymbol{C}(\underline{P})$ has three blocks, namely $135$, $2$ and $46$.

The limiting normality of $\cnt_{\underline{\sigma}}$, when $\underline{\sigma}$ is a vincular pattern was first established by Hofer \cite{hofer}. She proposes two different approaches to bound the Kolmogorov distance between the distribution of $\cnt_{\underline{\sigma}}$ and the Normal distribution, both based on dependency graphs. To apply them, she needs a lower bound for the variance of $\cnt_{\underline{\sigma}}$, i.e., to prove a more general version of Theorem \ref{th:boundBona}, which holds for any vincular pattern. Hofer obtained such a generalization by a rather complicated recurrence based on the law of total variance.
\begin{theorem}[Hofer, \cite{hofer}]
\label{th:hofer}
Let $X_{n}=\cnt_{\underline{\sigma}}$ be the number of occurrences of a vincular pattern $\underline{\sigma}$ with $j$ blocks, in a random permutation of length $n$. Then, there exists $c>0$, such that for all $n$,
\[
\Var(X_{n})\geq cn^{2j-1}.
\]
\end{theorem}
Below, we show that this more general bound is equivalent to a lemma generalizing Lemma ~\ref{lemma:classical}, that has an analogous interpretation as the one given with Theorem \ref{th:interpretation}. 

If $\underline{\sigma}$ is a vincular pattern of length $k$ with $j$ blocks, then we denote by $b_{\sigma}(m,j')$ the number of merges of two copies of $\underline{\sigma}$, where the resulting pattern is of length $m$ and has $j'$ blocks. 
\begin{example}[Merge of two copies of a vincular pattern]\leavevmode\\

Let $\underline{\sigma} = \underline{431}\,\underline{52}$. This pattern has length $k=5$ and $j=2$ blocks. Below is given a merge of two copies of $\underline{\sigma}$. The resulting pattern $\underline{6531}\,\underline{84}\,\underline{72}$ is of length $m=8$ and has $j'=3$ blocks.

\begin{table}[h!]
$\begin{array}{|c|c|c|c|c|c|c|c|}
\hline
6 & 5 & 3 & 1 & 8 & 4 & 7 & 2\\   \hhline{|=|=|=|=|=|=|=|=|}
6 & 5  & 3 &   & 8 & 4 &   &   \\ \hline
  & 5  & 3 & 1 &   &   & 7 & 2 \\ \hline
\end{array}$
\caption{Merge of two copies of the pattern $\underline{431}\,\underline{52}$.}
\label{fig:vincMergeExample}
\end{table}
\end{example}
\vspace{-6mm}
If $\underline{\sigma}$ has blocks of sizes $\alpha_{1},\ldots , \alpha_{j}$, then let $M_{\underline{\sigma}} = \max\limits_{1\leq i \leq j} \{\alpha_{i}\}$ and let $[x^{k}]P$ denotes the coefficient of the polynomial $P$ in front of $x^k$.

\begin{theorem}
\label{th:vinc}
Theorem \ref{th:hofer} is equivalent to 
\begin{equation}
\label{eq:lemmaVinc}
    \sum\limits_{l=1}^{M_{\underline{\sigma}}} (2k)_{l}b_{\sigma}(2k-l,2j-1) > \binom{2k}{k}\binom{2j-1}{j}j.
\end{equation}
\end{theorem}

\begin{proof}
We will use that the expected number of occurrences of a vincular pattern $\underline{\sigma}$ of length $k$, with $j$ blocks, in a random permutation of length $n$ is $\frac{\binom{n-(k-j)}{j}}{k!}$. This follows from Lemma ~\ref{lemma:ExpBivinc} and the fact that $|\boldsymbol{C}(\underline{\sigma})|=k-j$. Apply Corollary \ref{corr:main} and note that if $m_{1},m_{2}: \underline{\sigma},\underline{\sigma}\to \underline{P}$ and $\underline{P}$ has $2j-1$ blocks, then exactly one block of the first copy of $\underline{\sigma}$ was merged with one block of the second copy of $\underline{\sigma}$. Therefore, $|P|=\tilde{k}\in [2k-M_{\underline{\sigma}},2k]$ and $|\boldsymbol{C}(\underline{P})| = \tilde{k}-(2j-1)$. We have \\

\resizebox{0.97\hsize}{!}{
$
\begin{aligned}
\Var(X_{n}) = \mathbb{E}(X_{n}^{2}) -\mathbb{E}^{2}(X_{n}) =  \left[b(2k,2j)\frac{\binom{n-(2k-2j)}{2j}}{(2k)!} + \sum\limits_{l=1}^{M_{\underline{\sigma}}} b_{\sigma}(2k-l,2j-1)\frac{\binom{n-(2k-l-2j+1)}{2j-1}}{(2k-l)!}\right] \\
- \frac{\binom{n-(k-j)}{j}^{2}}{(k!)^{2}} + \mathcal{O}(n^{2j-2}). 
\end{aligned}$
}

We will again use that $\binom{n}{k} = \frac{(n)_{k}}{k!}$ and that $(n)_{k} = \sum\limits_{i=0}^{k}s(k,i)n^{i}$, where $s(k,i) = (-1)^{k-i}{ k\brack i}$ are the Stirling numbers of the first kind and ${ k\brack i}$ is the number of permutations in $S_{k}$ with $i$ disjoint cycles. Since ${ k\brack k-1} = \binom{k}{2}$ and $b_{\sigma}(2k,2j)=\binom{2k}{k}\binom{2j}{j} = \frac{(2j)!(2k)!}{(k!j!)^{2}}$, we get the following.

\resizebox{0.98\hsize}{!}{
$\begin{aligned}
\Var(X_{n}) & = \left[ \frac{(2j)!(2k)!}{(k!j!)^{2}}\frac{(n-2k+2j)_{2j}}{(2j)!(2k)!} + \sum\limits_{l=1}^{M_{\underline{\sigma}}} b_{\sigma}(2k-l,2j-1)\frac{(n-(2k-l-2j+1))_{2j-1}}{(2k-l)!(2j-1)!}\right] -
\frac{(n-k+j)_{j}^{2}}{(k!)^{2}(j!)^{2}} + \mathcal{O}(n^{2j-2}) \\
& = \frac{1}{(k!j!)^{2}}[(n-2k+2j)^{2j}-\binom{2j}{2}(n-2k+2j)^{2j-1}-
((n-k+j)^{j}-\binom{j}{2}(n-k+j)^{j-1}+\mathcal{O}(n^{j-2}))^{2}] \\
& +\sum\limits_{l=1}^{M_{\underline{\sigma}}} b_{\sigma}(2k-l,2j-1)\frac{(n-2k+l+2j-1)^{2j-1} + \mathcal{O}(n^{2j-2})}{(2k-l)!(2j-1)!}+ \mathcal{O}(n^{2j-2})\\
& = \frac{1}{(k!j!)^{2}}[n^{2j}-(2k-2j)n^{2j-1}-j(2j-1)n^{2j-1}-(n^{j}-(k-j)n^{j-1}-\frac{j(j-1)}{2}n^{j-1}+\mathcal{O}(n^{j-2}))^{2}] \\ & +\sum\limits_{l=1}^{M_{\underline{\sigma}}} b_{\sigma}(2k-l,2j-1)\frac{n^{2j-1}}{(2j-1)!(2k-l)!} + \mathcal{O}(n^{2j-2}).
\end{aligned}$
}\\

After simplifying, we get that $[n^{2j-1}]\Var(X_{n})>0$ if and only if
\begin{equation*}
\begin{split}
 \frac{-j^{2}}{(k!j!)^{2}}+ \sum\limits_{l=1}^{M_{\underline{\sigma}}}\frac{b_{\sigma}(2k-l,2j-1)}{(2k-l)!(2j-1)!} > 0 \Longleftrightarrow  \\
\sum\limits_{l=1}^{M_{\underline{\sigma}}} (2k)_{l}b_{\sigma}(2k-l,2j-1) > \binom{2k}{k}\binom{2j-1}{j}j.
\end{split}
\end{equation*}

\end{proof}
Note that when $j=k$, we have $M_{\underline{\sigma}}=1$ and $b_{\sigma}(2k-1,2j-1)=a_{\sigma}(2k-1)$, so we get Lemma \ref{lemma:classical}. When $j=1$, Inequality \eqref{eq:lemmaVinc} is trivial, since $M_{\underline{\sigma}}=k$ and on the left, just one of the summands (when $l=k$) is $(2k)_{k}$, while on the right we have $\binom{2k}{k}<(2k)_{k}$. We were not able to prove Inequality \eqref{eq:lemmaVinc} for vincular patterns with arbitrary number of blocks.

However, we can give an interpretation of this inequality. Note that when one merges two copies of a pattern with $j$ blocks and the obtained pattern has $2j-1$ blocks, then the blocks of the two copies can be aligned in exactly $\binom{2j-1}{j}j$ ways. These alignments will be called \emph{configurations}. For example, when $j=2$, there are $\binom{3}{2}2 = 6$ configurations shown below (the $\square$ symbol represents a block):

\renewcommand{\thefigure}{3}
\begin{figure}[ht!]
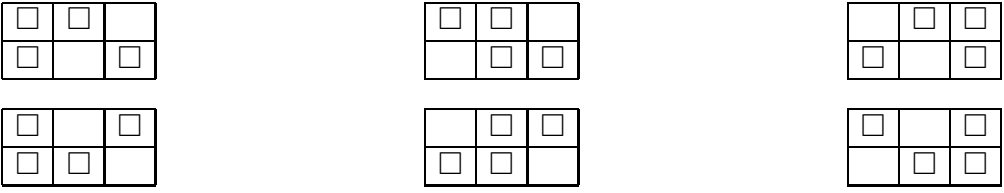

    \centering
    \begin{multicols}{3}
    \begin{itemize}[itemsep=10pt]
        \item[] $\begin{array}{|c|c|c|}
\hline
\square & \square &  \\ \hline
\square &   & \square\\ \hline
\end{array}$
\item[] $\begin{array}{|c|c|c|}
\hline
\square &   & \square\\ \hline
\square & \square &  \\ \hline
\end{array}$
\item[] $\begin{array}{|c|c|c|}
\hline
\square & \square & \\ \hline
 & \square &  \square \\ \hline
\end{array}$
\item[] $\begin{array}{|c|c|c|}
\hline
& \square &  \square \\ \hline
\square & \square  & \\ \hline
\end{array}$
\item[] $\begin{array}{|c|c|c|}
\hline
 & \square & \square  \\ \hline
\square &   & \square\\ \hline
\end{array}$
\item[] $\begin{array}{|c|c|c|}
\hline
\square &   & \square\\ \hline
& \square & \square  \\ \hline
\end{array}$
\end{itemize}
\end{multicols}
    \caption{The $6$ possible configurations, when merging two copies of a pattern with two blocks.}
    \label{fig:configs}
\end{figure}

For instance, the configuration corresponding to the merge shown on Figure \ref{fig:vincMergeExample} is the top-left configuration shown on Figure \ref{fig:configs}.
It is not difficult to see that the conjecture we give next would imply Inequality \eqref{eq:lemmaVinc} and respectively Theorem \ref{th:hofer} and the CLT for vincular patterns.
\begin{conjecture}
\label{conj:vinc}
For every vincular pattern $\sigma$ with $j$ blocks and every $1\leq l\leq M_{\underline{\sigma}}$,
\begin{equation}
\label{eq:conj}
b'_{\sigma}(2k-l,2j-1)>\frac{\binom{2k-l}{k}}{k_{(l)}}c_{\sigma ,l},
\end{equation}
 where $c_{\sigma ,l}\coloneqq$ the number of possible configurations for a merge of two copies of $\underline{\sigma}$, such that the minimum of the sizes of the two merged blocks is $l$ and $b'_{\sigma}(2k-l,2j-1)$ is the number of merges of two copies of $\underline{\sigma}$ with $l$ common elements and $2j-1$ blocks, such that they correspond to one of the same $c_{\sigma ,l}$ configurations.
\end{conjecture}
Indeed, it suffices to note that $\sum_{l=1}^{M_{\underline{\sigma}}}c_{\sigma ,l} = \binom{2j-1}{j}j$ and that $\frac{\binom{2k-l}{k}}{k_{(l)}} = \frac{\binom{2k}{k}}{(2k)_{l}}$. Thus, if we sum up Inequality \eqref{eq:conj} over $l$, we get Inequality \eqref{eq:lemmaVinc} with $b_{\sigma}$ replaced with $b'_{\sigma}$. Since $b_{\sigma}(2k-l,2j-1)\geq b'_{\sigma}(2k-l,2j-1)$, for all $l,j$ and $k$, Conjecture \ref{conj:vinc} would indeed imply Inequality \eqref{eq:lemmaVinc}. The ratio $\frac{\binom{2k-l}{k}}{k_{(l)}}$ is the expected number of merges when we fix one of the $c_{\sigma ,l}$ configurations and when we merge $\underline{\sigma}$ and $\underline{\sigma'}$, where $\sigma'\in S_k$ is a permutation selected uniformly at random and $\underline{\sigma'}$ has the same block structure as $\underline{\sigma}$. Therefore, Inequality \eqref{eq:conj} can be written as 
\begin{equation}
b'_{\sigma}(2k-s,2j-1)>\mathbb{E}(b'_{\sigma,\sigma'}(2k-s,2j-1)),
\end{equation}
where $b'_{\sigma,\sigma'}(2k-s,2j-1)$ is defined analogously to $a_{\sigma,\sigma'}(2k-1)$.
\subsection{Bivincular patterns}
\label{subsec:bivinc}
In the general case when $\underline{P}$ is a pattern for which $\boldsymbol{D}(\underline{P})$ might be non-empty, we do not necessarily have asymptotic normality of the distribution of $\cnt_{\underline{P}}$. For example, $\cnt_{\scaleto{\Vectorstack{\overline{12} \underline{21}}}{15pt}}$, which is the adjacency statistic $\adj$ introduced in Section \ref{sec:descents}, has Poisson distribution with mean $1$. This follows from a result proved by Wolfowitz \cite{wolfowitz} and independently by Kaplansky \cite{kaplansky} in the 1940s. They showed that if $X$ denotes the pairs of numbers $a$, $a+1$ that have consecutive positions in a permutation in $S_n$ that is chosen uniformly at random, then $X$ is asymptotically Poisson distributed with mean $2$. In 2014, Corteel et al. \cite{corteel} give another proof of this result that uses the method of Chen, which is used to prove convergence to Poisson distribution and which is an adaptation of the method of Stein for convergence to normal distribution \cite{stein}. Roughly, the method of Chen can be applied when one considers a sum of Bernoulli random variables such that many of them are independent. The article \cite{arratia} contains an accessible introduction and some good examples. 
 
Here, we reprove the fact that the asymptotic distribution of  $\adj$ is Poisson with mean $1$ by using Theorem \ref{th:adj} and the Fr\'{e}chet-Shohat Theorem given below.

\begin{theorem}[{\cite[Theorem 30.2]{momMeas}}]
\label{th:FrSh}
Suppose that the distribution of $X$ is determined by its moments and that $X_{n}$ have moments of all orders. Suppose also that $\lim_{n\to \infty}\mathbb{E}(X_{n}^{r}) = \mathbb{E}(X^{r})$, for $r=1,2,\ldots$. Then, $X_n$ converges in distribution to $X$.
\end{theorem}
\begin{definition}
The discrete random variable $X\coloneqq\Po(\lambda)$ is said to have a Poisson distribution, with parameter $\lambda>0$, if
\[
\mathbb{P}(X=k) = \frac{\lambda^{k}e^{-\lambda}}{k!},\qquad \text{for }k=0,1,\ldots
\]
\end{definition}
\begin{theorem}
As $n\to \infty$, $\adj$ converges in distribution to $\Po(1)$.
\end{theorem}
\begin{proof}
The Poisson measure is determined by its moments. One can deduce that using \cite[Theorem 30.1]{momMeas}. Because of Theorem \ref{th:FrSh}, it suffices to show that $\mathbb{E}(\adj^{r})$ converges to the $r$-th moment of $Po(1)$, when $n\to \infty$. A well-known fact is that the $r$-th moment of the Poisson distribution with mean $1$ is the $r$-th  Bell number $B_{r} = \sum\limits_{k=1}^{r}S(r,k)$, where $S(r,k)$ is the \emph{Stirling number of the second kind} (for more details, see \cite{pitman}). Looking at the double sum expression for $\mathbb{E}(\adj^{r})$ obtained in Theorem \ref{th:adj}, we see that every summand is a product of terms not including $n$ and the term $\frac{(\binom{n-(m-u)}{u})^{2}}{n_{(m)}}$ is $\mathcal{O}(1)$, unless $u = \frac{m}{2}$. Thus, when $n\to \infty$, we can look only at the terms corresponding to even values of $m$, i.e., $m=2m_{1}$ for some $m_{1} = 1,\ldots ,r$ and $u = \frac{m}{2} = m_{1}$. Since $\lim\limits_{n\to\infty} \frac{\binom{n-m_{1}}{m_{1}}^{2}}{n_{(2m_{1})}} = \frac{1}{(m_{1}!)^{2}}$ and $\sum\limits_{i=0}^{k}(-1)^{i}\binom{k}{i}(k-i)^{r} = k!S(r,k)$ , we obtain the following.
\begin{equation*}
\begin{split}
\lim_{n\to\infty} \mathbb{E}(\adj^{r}) = \sum\limits_{m_{1}=1}^{r}\left(\sum\limits_{w=0}^{m_{1}}(-1)^{w}\binom{m_{1}}{w}(m_{1}-w)^{r}\right)m_{1}!\frac{1}{(m_{1}!)^{2}} = \\ = \sum\limits_{m_{1}=1}^{r}\frac{\sum\limits_{w=0}^{m_{1}}(-1)^{w}\binom{m_{1}}{w}(m_{1}-w)^{r}}{m_{1}!} =  \sum\limits_{m_{1}=1}^{r} S(r,m_{1}) = B_{r}.
\end{split}   
\end{equation*}

\end{proof}
It would be interesting to investigate which are the possible asymptotic distributions of $\cnt_{\underline{P}}$ for other bivincular patterns? This question has been already stated in \cite[Section 1]{hofer}, where some approaches were also suggested.

\section{Patterns with linear valuation polynomials}
\label{sec:lin_val_polynomials}
In this section, we obtain direct formulas for $M(f_{\underline{P},Q},n)$ in three special cases using linearity of expectation. 
\begin{theorem}
\label{th:linearity}
Consider a simple statistic with a valuation polynomial $Q(s,w)=Q_1(s)Q_2(w)$. The following three formulas for $M(f_{\underline{P},Q},n)$ hold in the described cases for $Q_{1}$, $Q_{2}$, $|\boldsymbol{C}(\underline{P})|$ and $|\boldsymbol{D}(\underline{P})|$.

\captionsetup[table]{name=Table}
\begin{table}[h!]
\renewcommand\thetable{1}
\centering
\captionsetup{position=below}
\scalebox{0.7}{
\begin{tabular}{|c|c|c|c|c|c|} 
 \hline  
 \textbf{Formula} & \boldmath{$Q_1(s)$} & \boldmath{$Q_2(w)$} & \boldmath{$|C(P)|$} & \boldmath{$|D(P)|$} & \boldmath{$M(f_{\underline{P},Q},n)=\sum\limits_{\sigma\in S_{n}}\sum\limits_{s\in_{P}\sigma}Q_{1}(s)Q_{2}(w)$} \\ 
 \hline\hline
1 &
$a_{0} + a_{1}s_{1} + \ldots + a_{k}s_{k}$ &
1 &
c &
0 &
$\binom{n}{k}\binom{n-c}{k-c}(n-k)!(a_{0} + \frac{n+1}{k+1}\sum\limits_{i=1}^{k}ia_{i})$
\\
\hline
2 &
1 &
$b_{0} + b_{1}w_{1} + \ldots + b_{k}w_{k}$ &
0 &
d &
$\binom{n}{k}\binom{n-d}{k-d}(n-k)!(a_{0} + \frac{n+1}{k+1}\sum\limits_{j=1}^{k}jb_{j})$
\\
\hline
3 &
$a_{0} + a_{1}s_{1} + \ldots + a_{k}s_{k}$ &
$b_{0} + b_{1}w_{1} + \ldots + b_{k}w_{k}$ &
0 &
0 &
$\binom{n}{k}^{2}(n-k)!(a_{0} + \frac{n+1}{k+1}\sum\limits_{i=1}^{k}ia_{i})(b_{0} + \frac{n+1}{k+1}\sum\limits_{j=1}^{k}jb_{j})$
\\

 \hline
\end{tabular}
}
\caption{Formulas for $M(f_{\underline{P},Q},n)$ obtained with linearity of expectation in some special cases}
\label{table:M}
\end{table}
\end{theorem}

\begin{proof}
Consider the first row of Table \ref{table:M}. We have linear polynomial $Q_{1}$, constant $Q_{2}(w)=1$, $|\boldsymbol{C}(\underline{P})|=c$ for a given constant $c$ and $|\boldsymbol{D}(\underline{P})|=0$. One can write $M(f_{\underline{P},Q},n) =\\ \sum\limits_{\sigma\in S_{n}}\sum\limits_{s\in_{p}\sigma}Q(s,w)= \sum\limits_{\sigma\in S_{n}}\sum\limits_{s\in_{p}\sigma}Q_{1}(s) = \mathbb{E}(Q_{1}^{*})n!$, where $Q_{1}^{*}$ is a random variable defined over each $\sigma\in S_{n}$ as the sum of the $Q_{1}$-valuations for each occurrence of $\underline{P}$ in $\sigma$. Formally, $Q_{1}^{*}(\sigma) = \sum\limits_{s\in_{P}\sigma}Q_{1}(s)$. Let $v$ be the number of possible $k$-tuples of positions for the elements of the occurrences $s$ in $\sigma$, enumerated with $1,2,\ldots ,v$. We will use that
\[
Q_{1}^{*}(\sigma) = Q^{*}_{1,1}(\sigma) + Q^{*}_{1,2}(\sigma) + \cdots + Q^{*}_{1,v}(\sigma),
\]
where for each $j\in [v]$:
\[
Q^{*}_{1,j}(\sigma) \coloneqq 
\begin{cases}
Q_{1}(s),\text{ if $s\in_{P}\sigma$ and $s_{1},\ldots s_{k}$ appear in $\sigma$ at the $k$ positions indexed by $j$. }\\
0, \text{otherwise}.
\end{cases}
\]
We have $v = \binom{n-c}{k-c}$ since $|\boldsymbol{D}(\underline{P})| = 0$ and since for each $i\in \boldsymbol{C}(\underline{P})$, one can look at $P_{i}$ and $P_{i+1}$ as a single element. Due to symmetry, we have $\mathbb{E}(Q^{*}_{1,i}) = \mathbb{E}(Q^{*}_{1,j})$ for each $1\leq i,j\leq v$. Thus, using the linearity of expectation, we have
\[
\mathbb{E}(Q_{1}^{*}) = \sum\limits_{r=1}^{v}\mathbb{E}(Q_{1,r}^{*}) =  \binom{n-c}{k-c}\mathbb{E}(Q^{*}_{1,1}),
\]
where the $k$-tuple with number $1$ comprises the first $k$ possible positions for an occurrence of $\underline{P}$ in $\sigma$. We will show that
\begin{equation}
\label{eq:Q1}
\mathbb{E}(Q^{*}_{1,1}) = \frac{a_{0} + \frac{n+1}{k+1}\sum\limits_{i=1}^{k}ia_{i}}{k!}.    
\end{equation}

We have $\mathbb{E}(Q^{*}_{1,1}) = \frac{\mathbb{E}(Q_{1}(s))}{k!}$, where $s = (s_{1},\ldots ,s_{k})$ and $s_{1}< \cdots <s_{k}$ is a random $k$-subset of $[n]$. Since $Q(s) = a_{0} + a_{1}s_{1} + \cdots + a_{k}s_{k}$, we can use the linearity of expectation one more time to get $\mathbb{E}(Q_{1}(s)) = a_{0} + \sum\limits_{1}^{k} a_{i}\mathbb{E}(s_{i})$, where $s_{i}$ is the $i$-th ordered statistic for a $k$-sample without replacement from $[n]$. Therefore, 
\[
\mathbb{E}(s_{i}) = \sum\limits_{j=i}^{n-(k-i)}j\frac{\binom{j-1}{i-1}\binom{n-j}{k-i}}{\binom{n}{k}} = \frac{i}{\binom{n}{k}}\sum\limits_{j=i}^{n-(k-i)}\binom{j}{i}\binom{n-j}{k-i} = i\frac{\binom{n+1}{k+1}}{\binom{n}{k}} = i\frac{n+1}{k+1}.
\]
This establishes Equation \eqref{eq:Q1}.

Formula 2 for $M(f_{\underline{P},Q},n)$, under the conditions listed in the second row of Table \ref{table:M}, can be obtained in a similar way. To obtain Formula 3, a different transformation is used. In particular,
\begin{equation*}
 M(f_{\underline{P},Q},n) = \sum\limits_{\sigma\in S_{n}}\sum\limits_{s\in_{p}\sigma}Q(s,w) = \sum\limits_{\sigma\in S_{n}}\sum\limits_{s\in_{p}\sigma}Q_{1}(s)Q_{2}(w) = \mathbb{E}(Q_{1}^{*}Q_{2}^{*})n!,   
\end{equation*}
where $Q_{1}^{*}$ and $Q_{2}^{*}$ are random variables defined over $S_{n}$ as $Q_{1}^{*}(\sigma) = \sum\limits_{s\in_{P}\sigma}Q_{1}(s)$ and $Q_{2}^{*}(\sigma) = \sum\limits_{\substack{s\in_{P}\sigma \\ \sigma(w_{i})=s_{i}}}Q_{2}(w)$. Furthermore, let us also enumerate the possible $k$-tuples of values for an occurrence $s=(s_{1},\dots ,s_{k})$ with $s^{(1)},s^{(2)},\ldots ,s^{(\binom{n}{k})}$ and let $w^{(1)},w^{(2)},\ldots ,w^{(\binom{n}{k})}$ enumerates the possible $k$-tuples of positions for $s_{1},\dots ,s_{k}$ in an $n$-permutation $\sigma$. In addition, for $i,j\in [\binom{n}{k}]$, let $Q^{*}_{1,i,j}(\sigma) = Q_{1}(s^{(i)})$, if the values $s^{(i)}$ are at positions $w^{(j)}$ in $\sigma$ and let $Q^{*}_{1,i,j}(\sigma) = 0$ otherwise. Similarly, let $Q^{*}_{2,i,j}(\sigma) = Q_{2}(w^{(j)})$, if the values $s^{(i)}$ are at positions $w^{(j)}$ in $\sigma$ and let $Q^{*}_{2,i,j}(\sigma) = 0$ otherwise. Then, 
\begin{equation*}
    \mathbb{E}(Q_{1}^{*}Q_{2}^{*})n! = n!\sum\limits_{i=1}^{\binom{n}{k}}\sum\limits_{j=1}^{\binom{n}{k}}\mathbb{E}[Q^{*}_{1,i,j}Q^{*}_{2,i,j}] = n!\sum\limits_{i=1}^{\binom{n}{k}}\sum\limits_{j=1}^{\binom{n}{k}}\frac{1}{\binom{n}{k}k!}Q_{1}(s^{(i)})Q_{2}(w^{(j)}).
\end{equation*}
Thus, 
\begin{equation*}
    M(f_{\underline{P},Q},n) = (n-k)!\sum\limits_{i=1}^{\binom{n}{k}}\big(Q_{1}(s^{(i)})\sum\limits_{j=1}^{\binom{n}{k}}Q_{2}(w^{(j)})\big) = (n-k)!(\sum\limits_{i=1}^{\binom{n}{k}}Q_{1}(s^{(i)}))(\sum\limits_{j=1}^{\binom{n}{k}}Q_{2}(w^{(j)})).
\end{equation*}
Equation \eqref{eq:Q1} gives us $\sum\limits_{i=1}^{\binom{n}{k}}Q_{1}(s^{(i)}) = \binom{n}{k}\frac{a_{0} + \frac{n+1}{k+1}\sum\limits_{i=1}^{k}ia_{i}}{k!}$, $\sum\limits_{j=1}^{\binom{n}{k}}Q_{2}(w^{(j)}) = \binom{n}{k}\frac{b_{0} + \frac{n+1}{k+1}\sum\limits_{j=1}^{k}jb_{j}}{k!}$ and formula 3 follows.
\end{proof}
    \begin{example}
    $n=3$, $k=2$, $P=21$, $\boldsymbol{C}(\underline{P}) = \{1\}$, $\boldsymbol{D}(\underline{P}) = \emptyset$, $Q_{1}(s) = 3y_{1} + y_{2}$, $Q_{2}(w) = 1$. \\
    Then, we have $v = \binom{2}{1}$ possible sets of positions for an occurrence of the pattern, namely $(1,2)$ and $(2,3)$. One can readily check that $M(f_{\underline{P},Q},n) = \sum\limits_{\sigma\in S_{n}}\sum\limits_{s\in_{P}\sigma}Q_{1}(s) = 40$. Formula 1 in table \ref{table:M} indeed gives the same value.
        \end{example}
    
    To obtain $M(f_{\underline{P},Q},n)$ when $Q$ is of higher degree, one should be able to evaluate expectations of the kind $\mathbb{E}(s_{i_{1}}^{r_{1}}\cdots s_{i_{u}}^{r_{u}})$, where $\{s_{i_{1}},\ldots s_{i_{u}}\}$ is a random subset of $[n]$ with $u$ elements. To do that, one might use the theory of Ordered statistics (see \cite[Chapter 3.7]{ordStat}).

\section{Further questions}
We discuss three interesting further questions related to the results in the previous sections:
\setlist[itemize]{label=\rule[0.5ex]{0.6ex}{0.6ex}}
    \begin{itemize}
    \item Can we improve the bounds for the number of terms in Equation \eqref{T-ext-form}, Theorem \ref{th:simpleStat} and for the number of terms in Equation \eqref{eq:linearFormMain}, Theorem \ref{th:main}? Some computational evidence suggests that this might be possible.
    
    \item
     Can we prove the CLT for vincular patterns by giving either a combinatorial or algebraic proof to Equation \eqref{eq:lemmaVinc} in Theorem \ref{th:vinc}?
    
    \item Theorem \ref{th:main} shows that the aggregate of any permutation statistic is a linear combination of shifted factorials with constant coefficients. Similarly, in \cite{matchings}, Khare et al. showed that any statistic on matchings is a linear combination of double factorials with constant coefficients, whereas for statistics on the more general structure of set partitions, Chern et al. ~\cite{partitions} showed that we have linear combinations of shifted Bell numbers with polynomial coefficients. These facts suggest that most probably, there exists a combinatorial structure generalizing permutations, for which the aggregates of the statistics on it can be written as linear combinations of factorials with polynomial coefficients. Can we find such a structure, e.g., posets or polyominoes?   
    \end{itemize} 

\section*{Acknowledgement}
We are grateful to professor Catherine Yan and to professor Perci Diaconis for suggesting to us this approach for studying patterns in combinatorial structures. We are also thankful to Alexander Burstein for the helpful discussions over the results in Section \ref{sec:CLTs}.


\begin{thebibliography}{}
\bibitem{ordStat} Arnold, B. C., Balakrishnan, N., and Nagaraja, H. N. (2008). A first course in order statistics. Society for Industrial and Applied Mathematics.
\bibitem{arratia} Arratia, R., Goldstein, L. and Gordon, L., 1989. Two moments suffice for Poisson approximations: the Chen-Stein method. The Annals of Probability, 17(1), pp.9-25.
\bibitem{babson} Babson, E. and Steingr\'{i}msson, E., 2000. Generalized permutation patterns and a classification of the Mahonian statistics. S\'{e}m. Lothar. Combin, 44(B44b), pp.547-548.
\bibitem{baxter} Baxter, A. and Zeilberger, D., 2010. The Number of Inversions and the Major Index of Permutations are Asymptotically Joint-Independently Normal. arXiv preprint arXiv:1004.1160.
\bibitem{dyckPaths} Bernini, A., Ferrari, L., Pinzani, R. and West, J., 2013. Pattern-avoiding Dyck paths. In Discrete Mathematics and Theoretical Computer Science (pp. 683-694). Discrete Mathematics and Theoretical Computer Science.
\bibitem{bevan1324} Bevan, D., Brignall, R., Price, A.E. and Pantone, J., 2020. A structural characterisation of Av (1324) and new bounds on its growth rate. European Journal of Combinatorics, 88, p.103115.
\bibitem{momMeas} Billingsley, P., 2008. Probability and measure. John Wiley \& Sons.
\bibitem{bona} B\'{o}na, M. (2007). The copies of any permutation pattern are asymptotically normal. arXiv preprint arXiv:0712.2792.
\bibitem{bonaBook} B\'{o}na, M., 2012. Combinatorics of permutations. CRC Press.
\bibitem{HEC} B\'{o}na, M. ed., 2015. Handbook of enumerative combinatorics (Vol. 87). CRC Press.
\bibitem{jacopo} Borga, J., 2021. Asymptotic normality of consecutive patterns in permutations encoded by generating trees with one-dimensional labels. Random Structures \& Algorithms.
\bibitem{bivincular} Bousquet-M\'{e}lou, M., Claesson, A., Dukes, M., \& Kitaev, S. (2010). (2+2)-free posets, ascent sequences and pattern avoiding permutations. Journal of Combinatorial Theory, Series A, 117(7), 884--909.
\bibitem{burstein} Burstein, A. and H\"{a}st\"{o}, P., 2010. Packing sets of patterns. European Journal of Combinatorics, 31(1), pp.241-253.
\bibitem{Chern2} Chern, B., Diaconis, P., Kane, D. M., \& Rhoades, R. C. (2015). Central limit theorems for some set partition statistics. Advances in Applied Mathematics, 70, 92-105.
\bibitem{partitions} Chern, B., Diaconis, P., Kane, D. M., \& Rhoades, R. C. (2014). Closed expressions for averages of set partition statistics. Research in the Mathematical Sciences, 1(1), 2.
\bibitem{mallow} Crane, H., DeSalvo, S. and Elizalde, S., 2018. The probability of avoiding consecutive patterns in the Mallows distribution. Random Structures \& Algorithms, 53(3), pp.417-447.
\bibitem{corteel} Corteel, S., Louchard, G. and Pemantle, R., 2004. Common intervals of permutations. In Mathematics and Computer Science III (pp. 3-14). Birkh\"{a}user, Basel.
\bibitem{bridget2} Davis, R., Nelson, S.A., Petersen, T.K. and Tenner, B.E., 2018. The pinnacle set of a permutation. Discrete Mathematics, 341(11), pp.3249-3270.
\bibitem{trees} Dershowitz, N. and Zaks, S., 1989. Patterns in trees. Discrete Applied Mathematics, 25(3), pp.241-255.
\bibitem{lopez} Diaz-Lopez, A., Harris, P.E., Huang, I., Insko, E. and Nilsen, L., 2021. A formula for enumerating permutations with a fixed pinnacle set. Discrete Mathematics, 344(6), p.112375.
\bibitem{sagan} Domagalski, R., Liang, J., Minnich, Q., Sagan, B.E., Schmidt, J. and Sietsema, A., 2021. Pinnacle Set Properties. arXiv preprint arXiv:2105.10388.
\bibitem{elizaldeNoy} Elizalde, S. and Noy, M., 2003. Consecutive patterns in permutations. Advances in Applied Mathematics, 30(1-2), pp.110-125.
\bibitem{eriksen} Eriksen, N. and Sj\"{o}strand, J., 2011. Equidistributed statistics on matchings and permutations. arXiv preprint arXiv:1112.2120.
\bibitem{zohar} Even-Zohar, C., 2020. Patterns in random permutations. Combinatorica, 40(6), pp.775-804.
\bibitem{feller} Feller, W., 2008. An introduction to probability theory and its applications, vol 2. John Wiley \& Sons.
\bibitem{ferayEvans} F\'eray, V., 2013. Asymptotic behavior of some statistics in Ewens random permutations. Electronic Journal of Probability, 18, pp.1-32.
\bibitem{feray} F\'eray, V., 2020. Central limit theorems for patterns in multiset permutations and set partitions. The Annals of Applied Probability, 30(1), pp.287-323.
\bibitem{ferayDepGraphs} F\'eray, V., M\'eliot, P.L. and Nikeghbali, A., 2016. Dependency graphs and mod-Gaussian convergence. In Mod-$\phi$ Convergence (pp. 95-110). Springer, Cham.
\bibitem{fulman} Fulman, J., 2004. Stein's method and non-reversible Markov chains. In Stein's Method (pp. 66-74). Institute of Mathematical Statistics.
\bibitem{gaetz} Gaetz, C. and Ryba, C., 2020. Stable characters from permutation patterns. arXiv preprint arXiv:2006.04957.
\bibitem{goldstein} Goldstein, L., 2005. Berry-Esseen bounds for combinatorial central limit theorems and pattern occurrences, using zero and size biasing. Journal of applied probability, 42(3), pp.661-683.
\bibitem{hofer}Hofer, L., 2017. A central limit theorem for vincular permutation patterns. arXiv preprint arXiv:1704.00650.
\bibitem{hwang}Hwang, H. K., Chern, H. H., \& Duh, G. H. (2020). An asymptotic distribution theory for Eulerian recurrences with applications. Advances in Applied Mathematics, 112, 101960.
\bibitem{jansonBook} Janson, S., 1997. Gaussian hilbert spaces (No. 129). Cambridge university press.
\bibitem{janson1} Janson, S., 2017. Patterns in random permutations avoiding the pattern 132. Combinatorics, Probability and Computing, 26(1), pp.24-51.
\bibitem{janson2} Janson, S., 2019. Patterns in random permutations avoiding the pattern 321. Random Structures \& Algorithms, 55(2), pp.249-270.
\bibitem{JNZ} Janson, S., Nakamura, B. and Zeilberger, D., 2013. On the asymptotic statistics of the number of occurrences of multiple permutation patterns. arXiv preprint arXiv:1312.3955.
\bibitem{kammoun} Kammoun, M.S., 2020. Universality for random permutations and some other groups. arXiv preprint arXiv:2012.05845.
\bibitem{kaplansky} Kaplansky, I., 1945. The asymptotic distribution of runs of consecutive elements. The Annals of Mathematical Statistics, 16(2), pp.200-203.
\bibitem{matchings} Khare, N., Lorentz, R., \& Yan, C. H. (2017). Moments of matching statistics. Journal of Combinatorics, 8(1), 1--27.
\bibitem{kitaev} Kitaev, S. (2011). Patterns in permutations and words. Springer Science \& Business Media.
\bibitem{peaks} Ma, S.M., 2012. Derivative polynomials and enumeration of permutations by number of interior and left peaks. Discrete mathematics, 312(2), pp.405-412.
\bibitem{mann} Mann, H.B., 1945. On a test for randomness based on signs of differences. The Annals of Mathematical Statistics, 16(2), pp.193-199.
\bibitem{mansourBook} Mansour, T., 2013. Combinatorics of set partitions. Boca Raton: CRC Press.
\bibitem{mansour1324} Mansour, T. and Nassau, C., 2021. On Stanley-Wilf limit of the pattern 1324. Advances in Applied Mathematics, 130, p.102229.
\bibitem{petersen}Petersen, T. K. (2015). Eulerian numbers. In Eulerian Numbers (pp. 3-18). Birkh{\"a}user, New York, NY.
\bibitem{petersonGF} Petersen, T.K. and Guay-Paquet, M., 2014. The generating function for total displacement. arXiv preprint arXiv:1404.4674.
\bibitem{bridget} Petersen, T.K. and Tenner, B.E., 2012. The depth of a permutation. arXiv preprint arXiv:1202.4765.
\bibitem{pitman} Pitman, J., 1997. Some probabilistic aspects of set partitions. The American mathematical monthly, 104(3), pp.201-209.
\bibitem{rusu} Rusu, I. and Tenner, B.E., 2021. Admissible pinnacle orderings. Graphs and Combinatorics, pp.1-10.
\bibitem{stein} Stein, C., 1986. Approximate computation of expectations. IMS.
\bibitem{wolfowitz} Wolfowitz, J., 1944. Note on runs of consecutive elements. The Annals of Mathematical Statistics, 15(1), pp.97-98.
\bibitem{Z.I} Zeilberger, D., 2004. Symbolic moment calculus I: foundations and permutation pattern statistics. Annals of Combinatorics, 8(3), pp.369-378.
\bibitem{Z_moments} Zeilberger, D., 2009. The automatic central limit theorems generator (and much more!). In Advances in combinatorial mathematics (pp. 165-174). Springer, Berlin, Heidelberg.
\end{thebibliography}
\end{document}